%% file: main.tex
\documentclass{article}
\usepackage[a4paper,margin=1in]{geometry}  
\usepackage{amsmath, amssymb, amsthm, bm} 
\usepackage{graphicx}  
\usepackage{hyperref}  
\usepackage[numbers]{natbib}  
\usepackage{authblk}  
\usepackage{abstract}  
\usepackage{times}  
\usepackage{caption}
\usepackage{subcaption}

\input{commands}

\title{Improved error bounds for Koopman operator and reconstructed trajectories approximations with kernel-based methods}
\author[1]{Diego Olguín}
\author[1]{Axel Osses}
\author[1]{Héctor Ramírez}
\affil[1]{Departamento de Ingeniería Matemática and Centro de Modelamiento Matemático (CNRS IRL2807), Universidad de Chile, Santiago, Chile}
\date{}  

\begin{document}
\maketitle

\begin{abstract}
In this article, we propose a new error bound for Koopman operator approximation using Kernel Extended Dynamic Mode Decomposition. The new estimate is $O(N^{-1/2})$, with a constant related to the probability of success of the bound, given by Hoeffding's inequality, similar to other methodologies, such as Philipp et al. Furthermore, we propose a \textit{lifting back} operator to obtain trajectories generated by embedding the initial state and iterating a linear system in a higher dimension. This naturally yields an $O(N^{-1/2})$ error bound for mean trajectories. Finally, we show numerical results including an example of nonlinear system, exhibiting successful approximation with exponential decay faster than $-1/2$, as suggested by the theoretical results. 
\end{abstract}

\tableofcontents

\section{Introduction}
\input{content/1_introduction} 

\section{Preliminaries}
\input{content/2_preliminaries}

\section{Main results}
\input{content/3_results}

\section{Numerical results}
\input{content/4_numericalresults}

\section{Discussion}
\input{content/5_discussion}

\section{Conclusion}
\input{content/6_conclusion}

\section*{Acknowledgements}
This work was supported by Center for Mathematical Modeling (CMM) Basal fund FB210005 for center of excellence, FONDECYT 1240200, FONDAP/1523A0002, ECOS240038 and DO210001 all from ANID-Chile, and Joint Cooperation Fund Chile-Mexico 2022: 'Mathematical
Modeling of Epidemic Processes, Incorporating Population Structure, Regional Distribution and
Risk Groups'.

\bibliographystyle{plainnat}
\bibliography{references}  

\end{document}

%% file: commands.tex
\newcommand{\R}{\mathbb{R}} 
\newcommand{\N}{\mathbb{N}} 

\newcommand{\E}{\mathbb{E}} 
\renewcommand{\P}{\mathbb{P}} 

\newcommand{\norm}[1]{\left\lVert#1\right\rVert} 
\newcommand{\inner}[2]{\left\langle#1, #2\right\rangle} 

\newcommand{\eps}{\varepsilon} 

\newtheorem{defn}{Definition}
\newtheorem{theo}{Theorem}
\newtheorem*{theo*}{Theorem}
\newtheorem{prop}{Proposition}
\newtheorem{lem}{Lemma}

\newtheorem{cor}{Corollary}
\newtheorem{assu}{Assumption}

\renewcommand{\H}{\mathcal{H}}
\newcommand{\U}{\mathcal{U}}

\newcommand{\B}{\mathcal{B}}
\newcommand{\X}{\mathcal{X}}

%% file: content/1_introduction.tex
The study of dynamical systems, both in continuous and discrete time, is a widely developed area of research in various branches of mathematics, science, and engineering. In his pioneering work, Koopman \cite{Koopman1931HamiltonianSpace} introduced the Koopman operator, a fundamental mathematical tool for the study of dynamical systems, which was later extended by Koopman and von Neumann \cite{Koopman1932DynamicalSpectra}.

In recent years, the Koopman operator has re-emerged in the literature as a key tool for analyzing nonlinear dynamical systems through the lens of finite-data approximations. This renewed interest has been driven primarily by the work of Mezi\'c and Budi\v{s}i\'c, along with other collaborators \cite{Budisic2009AnObservables, Budisic2012GeometryFlows, Budisic2012AppliedKoopmanism}, giving rise to what many authors recognize as 'Koopmanism', a line of research continued by Mezi\'c in collaboration with several co-authors \cite{Arbabi2017StudyDecomposition,Mezic2013AnalysisOperator,Mezic2022OnOperator,Mezic2024ACases}. Another important research group on the Koopman operator was led by Brunton, Kutz and Proctor \cite{Baddoo2022KernelOptimization, Brunton2022ModernSystems, Proctor2016KoopmanControl, Proctor2018GeneralizingControl}.

Exploiting the linearity of the Koopman operator, albeit at the cost of working in infinite-dimensional spaces, various techniques have been developed to approximate it using finite-dimensional matrices, thereby reformulating nonlinear problems in computationally accessible terms. In this context, the works of Schmid \cite{Schmid2008DynamicData}, who introduced \textit{Dynamic Mode Decomposition} (DMD), and Williams \cite{Williams2015ADecomposition}, who proposed \textit{Extended Dynamic Mode Decomposition} (EDMD), stand out.

The convergence of EDMD techniques has been studied qualitatively in \cite{Mezic2022OnOperator, Klus2016OnOperator, Williams2015ADecomposition} and quantitatively in \cite{Kohne2025boldsymbolLboldsymbolinftyDecomposition, Nuske2023Finite-DataControl, Philipp2023ErrorFramework}. In particular, in \cite{Zhang2022APredictions} it was shown for deterministic systems under certain assumptions that the error bound is of order $O(N^{-1/2})$, where $N$ is the number of data points used to approximate the operator.

An important work on error bounds for Koopman operator approximation is that of Philipp et al. \cite{Philipp2024ErrorOperator}, where they proved, in the case of stochastic systems, a bound in probability in terms of $N \in \N$ the number of points used for approximate the operator that is of the form
\[
\norm{\U-\U_N}_{\H_{N} \to L^2} \leq \sqrt{\lambda_{N+1}} \norm{\U}_{\H \to \H} + c_N \eps,
\]
where $\U$ is the Koopman operator, $\U_N$ is its approximation by EDMD, $\lambda_{N+1}$ is a constant decreasing to $0$ as $N$ goes to infinity, $c_N$ is a constant depending on $N$, and $\varepsilon > 0$ is sufficiently small. This result is important in the recent literature because it gives a bound that quantifies the convergence of $\U_N$ to $\U$ in terms of the number of points, including an arbitrary error $\varepsilon$.

This key result has two important possibilities for improvement, the first being the dependence of the probability of success on $\varepsilon$ and this on $N$. It will be shown later in this article that in order to obtain an error bound of order $O(N^{-1/2})$, it's necessary that the probability of success decreases to $0$ as $N$ tends to infinity. The other possibility to improve the bound is about an assumption made in Philipp et al. that the eigenvalues of a certain operator are simple, which is not easy to confirm in some systems, so it's important to give an example or remove it.

In this work, we propose an error bound for the approximation of the operator for stochastic systems, of order $O(N^{-1/2})$, the same as other mentioned works, but with a high probability of success for the convergence of EDMD, and without requiring the eigenvalue simplicity assumption. This leads to a universal framework for approximating the Koopman operator without compromising the precision order of $O(N^{-1/2})$, an expected order of error for this type of approximation.

To achieve this, the paper is structured as follows: in Section 2, we present the necessary preliminaries for understanding the Koopman approximation in Reproducing Kernel Hilbert Spaces (RKHS) for stochastic systems. Section 3 contains our main theoretical results, first for the Koopman operator approximation in RKHS and then for the trajectory approximation using a lifting-back operator proposed in this work. In Section 4, we present some numerical results for linear systems and SIR systems.  

%% file: content/2_preliminaries.tex
In this section, we present the necessary notation and preliminaries for the results of this work, dividing it into Reproducing Kernel Hilbert Spaces (RKHS) notation and key results, covariance operators, and the Koopman operator.

Consider an autonomous and stochastic dynamical system represented by the
by the following equation:
\begin{equation}
    \mathbf{x}_{k+1} = \mathbf{f}\left(\mathbf{x}_k, \mathbf{w}_k\right).
    \label{eq:dyn_sys}
\end{equation}
Here, $\mathbf{f}: \R^n \times \Omega_\mathbf{w} \to \R^n$, where:
\begin{itemize}
    \item $\mathbf{x}_k$ represents the state of the system at the discrete time step
$k$.
    \item $\mathbf{w}_k$ is a random variable with support equal to
$\Omega_\mathbf{w}$, which introduces stochasticity into the system.
\end{itemize}

We can also assume that $\mathbf{x}_{k+1}$ depends on time and external inputs or controls. In general, however, these dependencies can be incorporated as extra states, allowing a system with such dependencies to be treated as one of the form considered.

Although the functional form used to represent the dynamical system in \eqref{eq:dyn_sys} is better for modeling and understanding the system, a probabilistic view of the dynamical system through its transition measure is more appropriate for demonstrating the desired results. This measure models the transition from a point to a set. Consider the Borel \(\sigma\) algebra over $\R^n$, denoted by $\B$.

\begin{defn}[Transition measure] \label{def:transition_measure} The \textit{transition measure} is defined as the probability measure
\(\rho: \R^n \times \B \to [0, 1]\), which acts as
\[
\rho(\mathbf{x}, \mathcal{A}) = \P (\mathbf{f}(\mathbf{x}, \cdot) \in \mathcal{A}).
\]
We use the notation \(\rho (\mathbf{x}, dx') = d\rho(\mathbf{x}, \cdot) (x')\).
\end{defn}
A property that will be needed later is that all points of the dynamical system remain within a set with probability $1$. This set is called the invariant space of the dynamical system.

\begin{defn}[Invariant space] A subset $\X \subseteq \R^n$ is called an \textit{invariant space} if 
\[
\rho (\mathbf{x}, \X) = 1, \, \forall \mathbf{x} \in \X.
\]
\end{defn}

\subsection{Reproducing Kernel Hilbert Spaces}

Reproducing Kernel Hilbert Spaces will be functional spaces in which the Koopman operator satisfies favorable properties under the assumptions of the dynamical system. These spaces are characterized by the existence of a function, called the reproducing kernel, which has good properties with respect to the inner product of the Hilbert space.

For completeness, we present some fundamental results from the theory of RKHS that are crucial for understanding the definition of the Koopman operator and the subsequent results in this paper.

\begin{defn}[Reproducing Kernel Hilbert Space (RKHS) \cite{Mercer1909XVI.Equations}]
A Hilbert space \( \H \subset \mathbb{C}^{\X} \) is said to be an RKHS if there exists a function \( k: \X \times \X \to \mathbb{C} \), called the \textit{reproducing kernel}, such that:
\begin{enumerate}
    \item \( k_p \equiv k(\cdot, p) \in \H, \ \forall p \in \X \).
    \item \( f(p) = \langle f, k_p \rangle_{\H}, \, \forall p \in \X, \ \forall f \in \H. \)
\end{enumerate}
\end{defn}
The second property, known as the \textit{reproducing property}, is fundamental in RKHS theory and gives rise to many important properties, as will be presented next. In the following, $\H$ will denote a RKHS asociated to a kernel $k$, that in general is assumed to be a positive definite function.

\begin{defn}[Positive Semi-Definite Function]
A function \( k: \X \times \X \to \mathbb{C} \) is called positive semi-definite if for all \( n \geq 1 \), \( (c_1, \dots, c_n) \in \mathbb{C}^n \), and \( (x_1, \dots, x_n) \in \X^n \), it holds that:
\[
\sum_{i=1}^n \sum_{j=1}^n c_i \overline{c_j} k(x_i, x_j) \geq 0.
\]
If the inequality is strict, the function is said to be positive definite.
\end{defn}

One possible way to understand the functions in these spaces is as pointwise limits of functions defined as partial series of terms related to finite data sets, as the following theorem states. This fact will be used in the proofs of the results in this work.

\begin{theo}[Moore–Aronszajn \cite{Aronszajn1950TheoryKernels}]
Let \( k \) be a positive semi-definite function on \( \X \times \X \). Then, there exists a unique Hilbert space \( \H \subset \mathbb{C}^{\X} \) with \( k \) as its \textit{reproducing kernel}. The subspace \( \H_0 \) of \( \H \), generated by the functions \( (k(\cdot, x))_{x \in \X} \), is dense in \( \H \), and \( \H \) coincides with the set of functions on \( \X \) that are pointwise limits of Cauchy sequences in \( \H_0 \), with the inner product defined by
\[
\langle f, g \rangle_{\H_0} = \sum_{i=1}^n \sum_{j=1}^m \alpha_i \bar{\beta}_j k(y_j, x_i),
\]
where, for \( (x_1, \dots, x_n) \in \X^n \) and \( (y_1, \dots, y_m) \in \X^m \) we have
\[
 f(\cdot) = \sum_{i=1}^n \alpha_i k(\cdot, x_i), \quad g(\cdot) = \sum_{j=1}^m \beta_j k(\cdot, y_j).
\]

\end{theo}

A key object that will appear next is the \textit{feature map}, which can be understood as a function that maps $\X$ into $\ell^2$, providing an embedding that represents the kernel as an inner product in $\ell^2$.

\begin{theo}
A function \( k: \X \times \X \to \mathbb{C} \) is a \textit{reproducing kernel} of a Hilbert space \( \H \) if and only if there exists a function \( \Phi: \X \to \ell^2(\X) \) such that, for all \( (x, y) \in \X \times \X \),
\[
k(x, y) = \langle \Phi(x), \Phi(y) \rangle_{\ell^2(\X)}.
\]
Moreover, the space \( \ell^2(\X) \) is isometric to \( \H \) via an isometry \( T: \H \to \ell^2(\X) \), and the function \( \Phi \) is given by \( \Phi(x) = T(k(\cdot, x)) \). This function \( \Phi \) is called the \textit{feature map} associated with \( k \).
\end{theo}

An assumption that will be made later is that the RKHS generated by a kernel on \( \X \) is norm-equivalent to a certain Sobolev space. To this end, we introduce a kernel that is closely related to fractional Sobolev spaces: the Matérn kernel.

\begin{defn}
The Mat\'ern kernel with smoothness parameter \( \nu > 0 \) and length scale \( \ell > 0 \) is defined as
\[
k_{\nu}(x,y) = \frac{2^{1-\nu}}{\Gamma(\nu)} \left( \sqrt{2\nu} \frac{\|x-y\|}{\ell} \right)^\nu B_\nu \left( \sqrt{2\nu} \frac{\|x-y\|}{\ell} \right),
\]
where \( \Gamma \) denotes the Gamma function and \( B_\nu \) is the modified Bessel function of the second kind with parameter \( \nu \).
\end{defn}

To explore the connection between Sobolev spaces and this kernel, we introduce key concepts, such as the definition of Sobolev spaces on \( \mathbb{R}^n \) using the Fourier transform, and the interior cone condition.

\begin{defn}[Fractional Sobolev Space \cite{Adams2003SobolevSpaces}]
The Sobolev space of smoothness \( s > 0 \) is defined as
\[
H^s(\mathbb{R}^n) = \left\{ f \in L^2(\mathbb{R}^n) : \widehat{f}(\cdot)(1 + \|\cdot\|_2^2)^{s/2} \in L^2(\mathbb{R}^n) \right\},
\]
where \( \widehat{f} \) denotes the Fourier transform of \( f \).
\label{def:frac_sob}
\end{defn}

\begin{defn}[Interior Cone Condition \cite{Wendland2004ScatteredApproximation}]
A set \( \Omega \subset \mathbb{R}^n \) satisfies the interior cone condition if there exist an angle \( \theta \in (0, \pi/2) \) and a radius \( r > 0 \) such that, for each \( x \in \Omega \), there exists a unit vector \( \xi(x) \in \mathbb{R}^n \) with \( \|\xi(x)\|_2 = 1 \) such that the cone
\[
C(x, \xi(x), \theta, r) := \left\{ x + \lambda y : y \in \mathbb{R}^n, \|y\| = 1, y^\top \xi(x) \geq \cos \theta, \lambda \in [0, r] \right\}
\]
is contained in \( \Omega \).
\end{defn}

The following theorem establishes a key condition for the developments in this work, that the Mat\'ern kernel is norm-equivalent to a Sobolev space.

\begin{theo}[\cite{Tuo2016AProperties, Wendland2004ScatteredApproximation}]
If \( \X \) is a compact set that satisfies the interior cone condition and \( k_\nu \) is the Mat\'ern kernel with parameter \( \nu > 0 \), then for \( s = \nu + n/2 \), the RKHS \( \H \) associated with \( k_\nu \) is norm-equivalent to the Sobolev space \( H^s(\X) \).
\end{theo}

Although the definition of the Matérn kernel appears complex and computationally expensive due to the Bessel function, it simplifies considerably for certain values of \( \nu \). In these cases, the kernel reduces to a function involving only fundamental functions, like exponentials and polynomials. The following proposition presents the case \( \nu = 1/2 \), while the cases \( \nu = 3/2 \) and \( \nu = 5/2 \) also yield closed-form expressions \cite{Porcu2024TheLearning}.

\begin{prop} \label{prop:laplace_kernel}
If \( \nu = 1/2 \), then
\[
k_\nu (x, y) = \exp \left( -\frac{\|x - y\|}{\ell} \right).
\]
\end{prop}

\begin{proof}[Proof of Proposition \ref{prop:laplace_kernel}]
From \cite{Barton1965HandbookTables., Davis1944AFunctions}, it is known that
\[
B_{1/2}(z) = \sqrt{ \frac{\pi}{2} } \frac{e^{-z}}{\sqrt{z}},
\]
thus, we deduce that
\[
\begin{aligned}
\frac{\sqrt{2}}{\Gamma(1/2)} \left( \frac{\|x - y\|}{\ell} \right)^{1/2} \sqrt{ \frac{\pi}{2} } \exp \left( -\frac{\|x - y\|}{\ell} \right) \left( \frac{\|x - y\|}{\ell} \right)^{-1/2} = \exp \left( -\frac{\|x - y\|}{\ell} \right).
\end{aligned}
\]
Obtaining the desired expression for \(k_\nu(x, y) \).
\end{proof}

We consider \( \X \) an invariant space of the dynamical system, and denote by \( \B_\X \) the trace \( \sigma \)-algebra of \( \B \) on \( \X \). We define \( \mu_\X: \B_\X \to [0,1] \) as a probability measure representing a random variable \( X \), where the data will be sampled. Also, we denote $X^+$ the random variable that, given $X = \mathbf{x}$, is represented by the measure $\rho(\mathbf{x}, \cdot)$, this models the posible states to which the system could evolve from a point $\mathbf{x}$. Let \( \Phi \) denote the canonical feature map associated with the kernel, that is, \( \Phi(x) = k(x, \cdot) \).

\subsection{Covariance and Koopman Operators}

From the theory of RKHS, covariance operators arise as key objects in our framework. To introduce them, we first define the Kronecker product in RKHS as a generalization of the Kronecker product of matrices. The results in this section follow from the works of Song, Fukumizu et al. \cite{Fukumizu2004DimensionalitySpaces, Fukumizu2013KernelKernels, Fukumizu2015NonparametricEmbedding, Song2009HilbertSystems, Song2013KernelModels}.

\begin{defn}[Kronecker Product]
The Kronecker product in an RKHS is defined as the following rank-one operator for fixed \( x, y \in \X \):
\[
\Phi(x) \otimes \Phi(y) : \H \to \H, \quad 
[\Phi(x) \otimes \Phi(y)] \psi = \langle \psi, \Phi(y) \rangle \Phi(x) = \psi(y)\, \Phi(x).
\]
\end{defn}

Taking the expectation of the embedded random variable \( X \) via the feature map defines an operator from \( \H \) to \( \H \), called the \textit{auto-covariance operator}.

\begin{defn}[Auto-covariance Operator]
The auto-covariance operator \( C_X : \H \to \H \) is defined as
\[
C_X = \E[\Phi(X) \otimes \Phi(X)] = \int_{\X} \Phi(x) \otimes \Phi(x) \, d\mu_\X(x).
\]
\end{defn}

We can also consider two random variables taking values in \( \X \), and embed each through the feature map. This leads to the definition of the \textit{cross-covariance operator}.

\begin{defn}[Cross-covariance Operator]
The cross-covariance operator associated with random variables \( X \) and \( X^+ \) is defined as the operator \( C_{XX^+} : \H \to \H \) given by
\[
C_{XX^+} = \mathbb{E} [\Phi(X) \otimes \Phi(X^+)] = \int_{\X} \int_{\X} \Phi(x) \otimes \Phi(y) \, \rho(x, dy) \, d\mu_\X(x),
\]
where \( \rho(x, dy) \) denotes the transition probability kernel from \( x \) to \( y \).
\end{defn}

The adjoint of these operators satisfies an elegant and useful identity, which will play a key role in the results of the next section.

\begin{prop} \label{prop:self_adjoint_kron}
For all \( x, y \in \X \), it holds that
\[
\left( \Phi(x) \otimes \Phi(y) \right)^* = \Phi(y) \otimes \Phi(x).
\]
\end{prop}

\begin{proof}[Proof of Proposition \ref{prop:self_adjoint_kron}]
Let \( h_1, h_2 \in \H \). Then, using the reproducing property, we compute
\[
\langle (\Phi(x) \otimes \Phi(y)) h_2, h_1 \rangle = \langle h_2(y) \Phi(x), h_1 \rangle = h_2(y) \langle \Phi(x), h_1 \rangle = h_2(y) h_1(x).
\]
On the other hand,
\[
\langle h_2, (\Phi(y) \otimes \Phi(x)) h_1 \rangle = \langle h_2, h_1(x) \Phi(y) \rangle = h_1(x) \langle h_2, \Phi(y) \rangle = h_1(x) h_2(y).
\]
Since the two expressions are equal, the identity follows.
\end{proof}

\begin{cor} \label{cor:adjoints_operator}
It holds that \( (C_{X})^* = C_X \) and \( (C_{XX^+})^* = C_{X^+X} \).
\end{cor}

\begin{proof}[Proof of Corollary \ref{cor:adjoints_operator}]
This follows directly from the linearity of the adjoint, the properties of the integral, and Proposition \ref{prop:self_adjoint_kron}:
\[
(C_X)^* = \left( \mathbb{E}[\Phi(X) \otimes \Phi(X)] \right)^* = \mathbb{E} \left[(\Phi(X) \otimes \Phi(X))^* \right] = \mathbb{E} [\Phi(X) \otimes \Phi(X)] = C_X.
\]
A similar argument applies to \( C_{XX^+} \).
\end{proof}

The error bounds derived in the next section depend on a constant related to the injectivity of the auto-covariance operator. This result follows from the spectral theorem \cite{Brezis2011FunctionalEquations}.

\begin{theo} \label{theo:CX_rank_selfad_compact}
The operator \( C_X: \H \to \H \) is an infinite-rank, self-adjoint, and compact operator.
\end{theo}

\begin{proof}[Proof of Theorem \ref{theo:CX_rank_selfad_compact}]
From the definition of \( C_X \), for all \( \psi \in \H \),
\[
C_X \psi = \int_\X \psi(x) \Phi(x) \, d\mu_\X(x),
\]
which implies that \( C_X \) has infinite rank. By Corollary \ref{cor:adjoints_operator}, we conclude that \( C_X \) is self-adjoint.

For compactness, let \( \{\psi_k\}_k \subset \H \) be a weakly convergent sequence with limit \( \psi \in \H \). To prove compactness, it suffices to show that \( \{C_X \psi_k\}_k \) converges strongly to \( C_X \psi \), i.e., that \( C_X \psi_k \to C_X \psi \) in norm. Observe that
\[
\begin{aligned}
    \| C_X \psi_k - C_X \psi \|_\H &= \left\| \int_\X (\Phi(x) \otimes \Phi(x)) (\psi_k - \psi) \, d\mu_\X(x) \right\|_\H \\ &\leq \int_\X | \psi_k (x) - \psi(x) | \|\Phi(x)\|_\H \, d\mu_\X(x)  \to 0.
\end{aligned}
\]
Thus, \( \| C_X \psi_k - C_X \psi \|_\H \to 0 \) as \( k \to \infty \), implying that \( \{C_X \psi_k\}_k \) converges in norm. Therefore, \( C_X \) is compact.
\end{proof}

\begin{cor}
    \label{cor:cx_inj}
    $C_X:\H \to \H$ is injective and there exists a positive constant $c_1$ such that \(c_1 \norm{\psi}_\H \leq \norm{C_X \psi}_\H, \, \forall \psi \in \H\).
    The constant $c_1$ will be called, in this work, the coercivity constant of $C_X$.
\end{cor}

\begin{proof}[Proof of Corollary \ref{cor:cx_inj}]
As $C_X$ is an infinite rank, self-adjoint, and compact operator, from the spectral theorem \cite{Brezis2011FunctionalEquations}, there exists an orthonormal basis of $\H$ composed of eigenfunctions of $C_X$, denoted by $\{\psi_k\}_k$, with positive eigenvalues $\{\lambda_k\}_k$ converging to $0$. Let $h \in \H$, then there exist $\{\alpha_k\}_k$ such that \(h = \sum_k \alpha_k \psi_k\). By linearity of $C_X$,
\[
\begin{aligned}
    \inner{C_Xh}{h} &= \sum_k \sum_j \alpha_j \alpha_k \inner{C_X \psi_k}{\psi_j} = \sum_k \sum_j \alpha_j \alpha_k \lambda_k \inner{\psi_k}{\psi_j} = \sum_k \alpha_k^2 \lambda_k,
\end{aligned}
\]
where the orthonormality of the basis was used, along with the fact that $\{\psi_k\}_k$ and $\{\lambda_k\}_k$ are eigenfunctions and eigenvalues of $C_X$, respectively. By Parseval's identity, \(\|h\|_\H^2 = \sum_k \alpha_k^2\), so it follows that if $\inner{C_Xh}{h} = 0$, then $h = 0$. In particular, if $C_X h = 0$, then $h=0$, so $C_X$ is injective.

The existence of the constant will be proven by contradiction. Suppose that for every $c_1 > 0$, there exists $\psi \in \H$ such that \(c_1 \| \psi \|_\H > \| C_X \psi\|_\H.\) Then, take a sequence $\{c_1^k\}_k \subset \R_+$ converging to $0$. So, there exists a sequence $\{\psi_k\}_k$ such that \(c_1^k \| \psi_k\|_\H > \| C_X \psi_k\|_\H.\) Without loss of generality, we can assume that $\| \psi_k \|_\H = 1$ for all $k$ (by linearity of $C_X$, except in the case that $\psi = 0$, which generates a contradiction). Then \(c_1^k > \| C_X \psi_k\|_\H\), so $\|C_X \psi_k\|_\H \to 0$. By the compactness of $C_X$ and the fact that $\{\psi_k\}_k$ is bounded (in norm 1), there exists a subsequence $\{\psi_{k_j}\}_j$ such that $C_X \psi_{k_j} \to 0$. By injectivity of $C_X$ and the fact that $C_X$ has a bounded inverse on its range (due to the positive eigenvalues bounded away from zero for the eigenfunctions in the range), this implies that $\psi_{k_j} \to 0$, which contradicts the fact that $\|\psi_{k_j}\|_\H = 1$ for every $j$. Therefore, such a constant $c_1 > 0$ must exist.
\end{proof}
Now, a fundamental tool, closely related to the Koopman operator, is the conditional embedding operator. This operator also represents a form of transition in the state phase, but by embedding the conditional expectation.

\begin{defn}[Conditional Embedding Operator]
    The conditional embedding operator between two distributions $X$ and $X^+$ is defined as $C_{X^+|X} : \H \to \H$ such that:
    \begin{enumerate}
        \item $\mu_{X^+|x} = \mathbb{E}_{X^+|X}[\Phi(X^+)|X=x] = C_{X^+|X}\Phi (x)$.
        \item $\mathbb{E}_{X^+|X}[h(X^+)|X=x] = \langle h, \mu_{X^+|x} \rangle$.
    \end{enumerate}
\end{defn}
A condition for the existence of this operator was given by Fukumizu et al. and Song et al. and is usually referred to as the Kernel Bayes Rule.

\begin{theo}[Kernel Bayes Rule \cite{Fukumizu2004DimensionalitySpaces, Song2009HilbertSystems}]
    Assuming that $ \mathbb{E}[h(X^+)|X = \cdot] \in \H $ for any $ h \in \H $, then:
    \[ C_{X^+|X} = C_{X^+X} C_{X}^{-1}.\]
\end{theo}
With all this, we can define the Koopman operator for autonomous and stochastic systems, similar to other works, e.g., \cite{Philipp2023ErrorFramework}. This operator is linear and represents the dynamics in a space of functions. Its linearity is one of the most valuable properties of the operator, and it has been exploited in the literature for the study of dynamical systems in recent years.

\begin{defn}[Koopman operator]
For a space of functions $\mathcal{F} \subset \mathbb{C}^\X$, the Koopman operator is defined as the operator \(\U : \mathcal{F} \to \mathbb{C}^\X \) acting as
\[
(\U \psi)(\mathbf{x}) = \E[\psi(\mathbf{f}(\mathbf{x},\cdot))] = \int_\X \psi(x') \rho (\mathbf{x}, dx').
\]
\end{defn}

In the case that $\mathcal{F}$ is a RKHS, the Koopman operator is strongly related to the covariance operators defined earlier, as presented in \cite{Philipp2024ErrorOperator}. This property depends on the preservation of the RKHS through the action of the Koopman operator. Later, we will provide some sufficient conditions for this to be accomplished.

\begin{prop} \label{prop:covariance_koopman}
    If $\U\H \subseteq \H$, we have that $C_{X^+|X} = \U^*$, which implies \(\U = C_X^{-1} C_{XX^+}\).
\end{prop}

\begin{proof}[Proof of Proposition \ref{prop:covariance_koopman}]
    First, an algebraic manipulation allows establishing the following relationship:
\begin{equation*}
    \begin{aligned}
        C_{X X^+} \psi = \int_\X \int_\X [\Phi (x) \otimes \Phi (y)] \psi \, d\rho (x, \cdot) (y) \, d\mu_\X (x) = \int_\X \int_\X \psi(y) \Phi (x) \, d\rho (x, \cdot) (y) \, d\mu_\X (x)
    \end{aligned}
\end{equation*}
where the definition of the Kronecker product was used. So,
\begin{equation*}
    \begin{aligned}
        C_{X X^+} \psi = \int_\X \int_\X \psi(y) \Phi (x) \, d\rho (x, \cdot) (y) \, d\mu_\X (x) = \int_\X (\U \psi)(x) \Phi (x) \, d\mu_\X (x) = C_X \U \psi.
    \end{aligned}
\end{equation*}
Therefore, it holds that \(C_{X X^+} = C_X \U.\) Noting that \(C_{X X^+} = \mathbb{E}[\Phi (X) \otimes \Phi (X^+)]\) and by virtue of Corollary \ref{cor:adjoints_operator}, it results that \(C_{X^+X} C_X^{-1} = C_{X^+ | X} = \U^*\).
\end{proof}

%% file: content/3_results.tex
This section is divided into two parts. The first part is devoted to establishing a result that quantifies the approximation of the Koopman operator \( \mathcal{U} \) by its empirical counterpart \( \mathcal{U}_N \), where \( N \) denotes the number of sampled points used in the approximation. Specifically, we derive a probabilistic error bound for this approximation.

\begin{theo}[Bound for kEDMD]
\label{theo:error_koop_sqrt_N_def}
Let \( \delta \in (0, 1) \) and \( N \in \mathbb{N} \) such that
\[
\delta > 2 \exp \left( -\frac{N c_1^2}{8\|k\|_\infty} \right),
\]
where \( c_1 \) is the coercivity constant of \( C_X \). Then, if \( \U \H \subset \H \), it holds with probability at least \( (1 - \delta)^2 \) that
\begin{equation*}
    \| \U - \U_N \|_{\H \to \H} \leq C_\delta N^{-1/2},
\end{equation*}
where
\[
C_\delta = \left( \frac{2}{c_1} + \frac{\sqrt{\|k\|_\infty}}{c_1^2} \right)
\sqrt{8 \|k\|_\infty \ln \left( \frac{2}{\delta} \right)}.
\]
\end{theo}

The second part aims to derive a more practical error bound from the result above. Let \( \mathbf{x}_k \) be a stochastic trajectory of the system and \( \mathbf{x}_{N,k} \) its approximation. Denote \( \hat{\mathbf{x}}_k \) the mean trajectory and \( \hat{\mathbf{x}}_{N,k} \) its empirical approximation. The following theorem establishes a probabilistic \( O(N^{-1/2}) \) bound on the approximation error of the expected trajectories.

\begin{theo} \label{theo:lifted_traj}
    Let $\delta \in (0, 1)$ and $N \in \mathbb{N}$ such that
\[
\delta > 2 \text{exp} \left ( -\frac{Nc_1^2}{8\|k\|_\infty}\right )
\]
where $c_1$ is the coercivity constant of $C_X$. Then, if $\U \H \subseteq \H$ and $\B^* (\R^p)^* \subseteq \H$, it holds with probability $(1-\delta)^3$ that there exists a constant $C_{\delta,k}$ such that
\[
\| \hat{\mathbf{x}}_{N,k} - \hat{\mathbf{x}}_k \| \leq C_{\delta,k} N^{-1/2}, \quad \left \| \E [\mathbf{x}_{N,k} - \mathbf{x}_k] \right \| \leq C_{\delta,k} N^{-1/2}.
\]
\end{theo}

To achieve the results of this work, some assumptions are needed, based in the  results seen in the preliminaries.

\begin{assu} \label{assu:assum1}
    The dynamics admits a compact invariant space $\X$ that is compact with Lipschitz boundary.
\end{assu}

\begin{assu} \label{assu:assum2}
    The transition measure defined in Definition \ref{def:transition_measure}, $\rho$, admits a density with respect the probability measure $\mu_\X$, $p:\X \times \X \to \R^+$ such that 
        \[
        \rho(x, \mathcal{A}) = \int_\mathcal{A} p(x,x') d\mu_\X(x'), \quad \forall x \in \X, \, \mathcal{A} \in \B_\X.
        \]
        And $p$ is $C^m$, with $m = \lceil n/2 + \nu \rceil$, for some $\nu > 0$. 
\end{assu}

\begin{assu} \label{assu:assum3}
    $k: \X \times \X$ is a kernel whose associated RKHS, $\H$, is norm-equivalent to $H^{n/2 + \nu}$, with $\nu$ the parameter called in Assumption \ref{assu:assum2}.
\end{assu}

An example where these assumptions are satisfied is a dynamical system with multiplicative noise. Consider the stochastic recurrence \(\mathbf{x}_{k+1} = \mathbf{x}_k \cdot \varepsilon_k\), where \(\mathbf{x}_0 \in (0, 1]\) and \(\varepsilon_k \sim \text{Unif}(0, 1)\). In this setting, the invariant space is \(\mathcal{X} = [0, 1]\), which is compact with a Lipschitz boundary. The reference measure \(\mu_{\mathcal{X}}\) is the Lebesgue measure on \([0, 1]\) (i.e., a uniform distribution), and the transition density is given by \(p(x, x') = 1/x\), since \(\mathbf{x}_{k+1} \sim \text{Unif}(0, \mathbf{x}_k)\). The only singular point is \(x = 0\), but this is attained with probability zero under the dynamics. The kernel can be taken as the Matérn kernel with smoothness parameter \(\nu = 1/2\), in which case the associated RKHS is norm-equivalent to the Sobolev space \(H^{1/2 + 1/2} = H^1\).

\subsection{Proof of Theorem \ref{theo:error_koop_sqrt_N_def}}

First, we present the kernel Dynamic Mode Decomposition (kEDMD) procedure for approximate the Koopman operator. We consider \( N \) points, represented as \( \{ x_i \}_{i=1}^N \sim \mu_\X^N \) and \( \{ x^+_i \}_{i=1}^N \) generated as:
\begin{equation*}
    x^+_i \sim \rho (x_i, \cdot), \quad i = 1, \dots, N.
\end{equation*}
We define the space \(\H_{N} = \text{span} \{\Phi (x_i) : i = 1, \dots, N \}\) and the following matrices are introduced:
\begin{equation*}
    \mathbf{X} = (x_{1} | \dots | x_N), \quad \mathbf{X}^+ = (x_{1}^+ | \dots | x_N^+),
\end{equation*}
\begin{equation*}
    \Phi_N (\mathbf{X}) = (k(x_i, x_j))_{i,j = 1}^N, \quad \Phi_N (\mathbf{X}^+) = (k(x_i, x^+_j))_{i,j = 1}^N.
\end{equation*}
Based on these definitions, we introduce the operators \(C_{X}^N : \H_N \to \H_N, \, C_{XX^+}^N : \H_N \to \H_N\) defined as
\[
C_X^N = \frac{1}{N} \sum_{j=1}^N \Phi (x_j) \otimes \Phi (x_j), \quad C_{XX^+}^N = \frac{1}{N} \sum_{j=1}^N \Phi (x_j) \otimes \Phi (x_j^+)
\]
Then, the actions of these operators are given by
\[
C_{X}^N \Phi (x_i) = \frac{1}{N} \sum_{j = 1}^N k(x_i, x_j) \Phi (x_j), \quad C_{XX^+}^N \Phi (x_i) = \frac{1}{N} \sum_{j = 1}^N k(x_i, x_j^+) \Phi (x_j).
\]
These operators are represented by the matrices \( \Phi_N (\mathbf{X}) \) and \( \Phi_N (\mathbf{X}^+) \), respectively. Next, we define the approximant operator of the Koopman operator \(\U_N : \H_N \to \H_N\) and its representant matrix, respectivly, as
\[
\U_N = (C_X^N)^{-1} C_{XX^+}^N, \, \mathbf{U}_N = (\Phi_N (\mathbf{X}))^{-1} \Phi_N (\mathbf{X}^+)^\top.
\]
An original result of this work is that the covariance operators can be extended continously using functions in $\H$ instead of $\H_N$, where the Moore-Aronszajn theorem is important.
\begin{prop} \label{prop:covariance_ext}
    The operators \( C_X^N \) and \( C_{XX^+}^N \) can be extended as continuous operators from \( \mathcal{H} \) to \( \mathcal{H}_N \).
\end{prop}

\begin{proof}[Proof of Proposition \ref{prop:covariance_ext}]
    Let \( \psi \in \mathcal{H} \) and \( \varepsilon > 0 \). By the Moore–Aronszajn theorem, there exists a finite set \( \{ \tilde{x}_j \}_{j=1}^m \subset \mathcal{X} \) such that 
    \[
        \psi_m := \sum_{j=1}^m \alpha_j \Phi(\tilde{x}_j) \in \mathcal{H}_{m} := \operatorname{span}\{ \Phi(\tilde{x}_j) \}_{j=1}^m
    \]
    satisfies \( \| \psi - \psi_m \|_{\mathcal{H}} \leq \varepsilon \). First, the operators are extended to finitely generated subspaces such as \( \mathcal{H}_{m} \), by defining their action on \( \Phi(x_i) \), which yields
    \[
        C_X^N \psi_m =
        \frac{1}{N} \sum_{i=1}^m \alpha_i \sum_{j = 1}^N k_{\mathcal{X}}(\tilde{x}_i, x_j) \Phi(x_j), \quad 
        C_{XX^+}^N \psi_m =
        \frac{1}{N} \sum_{i=1}^m \alpha_i \sum_{j = 1}^N k_{\mathcal{X}}(\tilde{x}_i, x_j^+) \Phi(x_j).
    \]
    The extension to \( \mathcal{H} \) is then defined as
    \[
        C_X^N \psi := \lim_{m \to \infty} C_X^N \psi_m, \qquad 
        C_{XX^+}^N \psi := \lim_{m \to \infty} C_{XX^+}^N \psi_m,
    \]
    which gives the expressions
    \[
        C_X^N \psi =
        \frac{1}{N} \sum_{i \geq 1} \alpha_i \sum_{j = 1}^N k(\tilde{x}_i, x_j) \Phi(x_j), \quad 
        C_{XX^+}^N \psi =
        \frac{1}{N} \sum_{i \geq 1} \alpha_i \sum_{j = 1}^N k(\tilde{x}_i, x_j^+) \Phi(x_j).
    \]
    To conclude, we verify that the operators are continuously defined:
    \[
        \| C_X^N \psi \|^2 
        = \frac{1}{N^2} \left\langle \sum_{i \geq 1} \alpha_i \sum_{j = 1}^N k(\tilde{x}_i, x_j) \Phi(x_j), \sum_{\ell \geq 1} \alpha_\ell \sum_{p = 1}^N k(\tilde{x}_\ell, x_p) \Phi(x_p) \right\rangle.
    \]
    By the reproducing property and boundedness of \(k\), this is bounded by
    \[
        \| C_X^N \psi \|^2 \leq \| k \|_\infty^2 \| \psi \|_{\mathcal{H}}^2.
    \]
    Similarly, we obtain \( \| C_{XX^+}^N \psi \|^2 \leq \| k \|_\infty^2 \| \psi \|_{\mathcal{H}}^2 \), which completes the proof.
\end{proof}

A key inequality used in our analysis is Hoeffding’s inequality in Hilbert spaces, which follows as a corollary of a martingale concentration result due to Pinelis~\cite{Pinelis1994OptimumSpaces}. This inequality is a concentration result that plays a central role in the analysis of empirical approximations and other statistical estimates in Hilbert spaces.

\begin{lem}[Hoeffding's inequality in Hilbert spaces, Pinelis~\cite{Pinelis1994OptimumSpaces}]
Let \( \xi_1, \dots, \xi_n \) be independent random variables in a separable Hilbert space \( H \), such that \( \mathbb{P} \)-almost surely \( \| \xi_i \|_H \leq M \), and \( \mathbb{E}[\xi_i] = 0 \) for all \( 1 \leq i \leq n \). Then, for every \( \varepsilon > 0 \),
\[
\mathbb{P} \left( \left\| \frac{1}{n} \sum_{i=1}^n \xi_i \right\|_H \geq \varepsilon \right) \leq 2 \exp \left( -\frac{n \varepsilon^2}{2M^2} \right).
\]
\end{lem}

An important consequence of our assumptions is that the Koopman operator preserves the RKHS structure. To establish this, we first prove that the Koopman operator preserves Sobolev spaces, extending a result previously established for deterministic systems by Köhne et al.~\cite{Kohne2025boldsymbolLboldsymbolinftyDecomposition} to the stochastic case.

\begin{prop} \label{prop:koopman_preserve}
Under Assumptions~\ref{assu:assum1}, \ref{assu:assum2}, and \ref{assu:assum3}, and for \( s = \lceil n/2 + \nu \rceil \), we have that \( \mathcal{U} H^s \subseteq H^s \), where \( H^s \) denotes the Sobolev space of order \( s \).
\end{prop}

The proof of Proposition~\ref{prop:koopman_preserve} relies on an interpolation theorem for Sobolev spaces due to Brenner and Scott~\cite{Brenner2008TheMethods}.

\begin{lem} \label{lema:interpolation}
Let \( \mathcal{X} \) be a compact domain with Lipschitz boundary, and let \( H^s(\mathcal{X}) \) denote the Sobolev space over \( \mathcal{X} \) with integrability index \( p = 2 \) and regularity \( s \). Then, for all \( s, k > 0 \),
\[
H^{s + \theta k}(\mathcal{X}) = [H^s(\mathcal{X}), H^{s + k}(\mathcal{X})]_\theta,
\]
for any \( \theta \in [0, 1] \). In particular, if \( \psi \in H^s(\mathcal{X}) \cap H^{s + k}(\mathcal{X}) \), then \( \psi \in H^{s + \ell}(\mathcal{X}) \) for all \( \ell \in [0, k] \).
\end{lem}

\begin{proof}[Proof of Proposition~\ref{prop:koopman_preserve}]
Let \( m \in \mathbb{N} \), \( m \leq k \), and let \( |\alpha| = m \) be a multi-index. For \( h \in H^m(\mathcal{X}) \), the dominated convergence theorem yields
\[
\partial_1^\alpha (\mathcal{U} h)(x) = \int_{\mathcal{X}} h(y) \, \partial_1^\alpha p(x, y) \, d\mu_{\mathcal{X}}(y).
\]
Hence,
\[
\| \partial_1^\alpha (\mathcal{U} h) \|_{L^2}^2 
\leq \int_{\mathcal{X}} \left( \int_{\mathcal{X}} |h(y)| \, |\partial_1^\alpha p(x, y)| \, d\mu_{\mathcal{X}}(y) \right)^2 d\mu_{\mathcal{X}}(x)
\leq \| \partial_1^\alpha p \|_{C^k}^2 \, \mu_{\mathcal{X}}(\mathcal{X}) \, \| h \|_{L^2}^2.
\]
This shows that \( \mathcal{U} h \in H^m(\mathcal{X}) \), with
\[
\| \mathcal{U} \|_{H^m(\mathcal{X}) \to H^m(\mathcal{X})} \leq \mu_{\mathcal{X}}(\mathcal{X})^{1/2} \sum_{|\alpha| \leq k} \| \partial_1^\alpha p \|_{C^k}.
\]
Therefore, \( \mathcal{U} H^m(\mathcal{X}) \subseteq H^m(\mathcal{X}) \) for all \( m \in \{0, \dots, k\} \). By Lemma~\ref{lema:interpolation}, for every \( \ell \in [0, k] \), we have that
\[
H^\ell(\mathcal{X}) = [H^0(\mathcal{X}), H^k(\mathcal{X})]_\theta, \quad \text{for } \theta = \ell / k,
\]
which implies \( \mathcal{U} H^\ell(\mathcal{X}) \subseteq H^\ell(\mathcal{X}) \). In particular, this holds for all \( \ell \in [0, s] \) since \( k = \lceil s \rceil \).
\end{proof}

Then, it follows directly from the norm equivalence between the RKHS and the Sobolev space that the Koopman operator preserves the RKHS.

\begin{cor} \label{cor:koopman_preserve_rkhs}
Under Assumptions~\ref{assu:assum1}, \ref{assu:assum2}, and \ref{assu:assum3}, it holds that \( \U \H \subseteq \H \).
\end{cor}

\begin{proof}[Proof of Corollary \ref{cor:koopman_preserve_rkhs}]
Let \( \psi \in \mathcal{H}_\mathcal{X} \). Since \( \mathcal{H}_\mathcal{X} \) is norm-equivalent to \( H^s(\mathcal{X}) \), there exists a constant \( C_s > 0 \) such that
\[
\| \mathcal{U} \psi \|_{\mathcal{H}_\mathcal{X}} \leq C_s \| \mathcal{U} \psi \|_{H^s(\mathcal{X})} 
\leq C_s \| \mathcal{U} \|_{H^s(\mathcal{X}) \to H^s(\mathcal{X})} \| \psi \|_{H^s(\mathcal{X})}
\leq C_s^2 \| \mathcal{U} \|_{H^s(\mathcal{X}) \to H^s(\mathcal{X})} \| \psi \|_{\mathcal{H}_\mathcal{X}}.
\]
Therefore, \( \mathcal{U} \psi \in \mathcal{H}_\mathcal{X} \), and in particular,
\[
\| \mathcal{U} \|_{\mathcal{H}_\mathcal{X} \to \mathcal{H}_\mathcal{X}} \leq C_s^2 \| \mathcal{U} \|_{H^s(\mathcal{X}) \to H^s(\mathcal{X})},
\]
which proves the result.
\end{proof}

The following lemma, proved by Philipp et al.~\cite{Philipp2024ErrorOperator}, will be instrumental for the propositions that follow. It provides an equivalence between the condition \( \U \H \subseteq \H \) and the relation between the ranks of the covariance and cross-covariance operators.

\begin{lem}[Philipp et al.~\cite{Philipp2024ErrorOperator}] \label{lemma: rank_inclusion}
The following statements are equivalent:
\begin{enumerate}
    \item \( \U \H \subseteq \H \),
    \item \( \U \in \mathcal{L}(\H, \H) \),
    \item \( \operatorname{Range}(C_{XX^+}) \subseteq \operatorname{Range}(C_X) \).
\end{enumerate}
\end{lem}

We proceed following the approach of Philipp et al.~\cite{Philipp2024ErrorOperator}, using Hoeffding's inequality to derive a probabilistic bound for the approximation of the covariance operators. Our result differs from theirs in that we employ the operator norm rather than the Hilbert--Schmidt norm.

\begin{prop}
\label{prop:app_covariances}
Let \( \varepsilon > 0 \) and \( N \in \mathbb{N} \). Then, with probability at least \( (1 - \delta)^2 \), we have
\[
\| C_{X} - C_{X}^N \| \leq \varepsilon, \quad \| C_{XX^+} - C_{XX^+}^N \| \leq \varepsilon,
\]
where
\[
\delta = 2 \exp\left( -\frac{N \varepsilon^2}{8 \| k \|_\infty} \right).
\]
\end{prop}

\begin{proof}[Proof of Proposition~\ref{prop:app_covariances}]
First, for any \( \psi \in \H \), we compute
\[
\| (\Phi(x_j) \otimes \Phi(x_j)) \psi \|^2 = \| \psi(x_j) \Phi(x_j) \|^2 \leq \| \psi \|_\infty^2 k(x_j, x_j) \leq \| k \|_\infty \| \psi \|_\H^2,
\]
which implies \( \| \Phi(x_j) \otimes \Phi(x_j) \| \leq \sqrt{ \| k \|_\infty } \). Similarly, since \( \| \Phi(x_j^+) \| \leq \sqrt{ \| k \|_\infty } \), we also have
\[
\| \Phi(x_j) \otimes \Phi(x_j^+) \| \leq \sqrt{ \| k \|_\infty }.
\]
Taking expectations, we obtain
\[
\| \E[ \Phi(x_j) \otimes \Phi(x_j) ] \| \leq \E[ \| \Phi(x_j) \otimes \Phi(x_j) \| ] \leq \sqrt{ \| k \|_\infty },
\]
\[
\| \E[ \Phi(x_j) \otimes \Phi(x_j^+) ] \| \leq \E[ \| \Phi(x_j) \otimes \Phi(x_j^+) \| ] \leq \sqrt{ \| k \|_\infty }.
\]
It follows that
\[
\| \Phi(x_j) \otimes \Phi(x_j) - C_X \| \leq 2 \sqrt{ \| k \|_\infty }, \quad \| \Phi(x_j) \otimes \Phi(x_j^+) - C_{XX^+} \| \leq 2 \sqrt{ \| k \|_\infty }.
\]

Now, let \( \varepsilon > 0 \) and define
\[
\delta = 2 \exp\left( -\frac{N \varepsilon^2}{8 \| k \|_\infty} \right).
\]
Applying Hoeffding's inequality in Hilbert spaces, we get
\[
\mathbb{P}\left( \| C_X - C_X^N \| > \varepsilon \right) \leq \delta, \quad 
\mathbb{P}\left( \| C_{XX^+} - C_{XX^+}^N \| > \varepsilon \right) \leq \delta.
\]
Therefore, with probability at least \( (1 - \delta)^2 \), we conclude
\[
\| C_X - C_X^N \| \leq \varepsilon, \quad \| C_{XX^+} - C_{XX^+}^N \| \leq \varepsilon.
\]
\end{proof}

Now we present the proof of Theorem \ref{theo:error_koop_sqrt_N_def} for kEDMD, where the injectivity of $C_X$ is critical, using the coercivity constant. This is the most important difference with the work of Philipp et al.  

\begin{proof}[Proof of Theorem \ref{theo:error_koop_sqrt_N_def}]
    Let $\psi \in \H$, then
\[
\begin{aligned}
    \| \U \psi - \U_N \psi \| &= \| C_X^{-1} C_{XX^+} \psi - \left (C_X^N \right )^{-1} C_{XX^+}^N \psi \| \\
    & \leq \| C_X^{-1} C_{XX^+} \psi - C_X^{-1} C_{XX^+}^N \psi \| + \| C_X^{-1} C_{XX^+}^N \psi - \left (C_X^N \right )^{-1} C_{XX^+}^N \psi \|.
\end{aligned}
\]
For the first term: \(\| C_X^{-1} C_{XX^+} \psi - C_X^{-1} C_{XX^+}^N \psi \|\),
let \(\Tilde{\psi} = C_{XX^+} \psi, \,  \Tilde{\psi}_N = C_{XX^+}^N \psi\) so that
\[
\| C_X^{-1} C_{XX^+} \psi - C_X^{-1} C_{XX^+}^N \psi \| = \| C_X^{-1} \Tilde{\psi} - C_X^{-1} \Tilde{\psi}_N \|,
\]
then let $\hat{\psi}$ and $\hat{\psi}_N$ such that \(C_X \hat{\psi} = \Tilde{\psi}, \, C_X \hat{\psi}_N = \Tilde{\psi}_N\), 
which exist since $\text{Ran}(C_{XX^+}) \subseteq \text{Ran}(C_X)$. Thus,
\[
\| C_X^{-1} C_{XX^+} \psi - C_X^{-1} C_{XX^+}^N \psi \| = \|\hat{\psi} - \hat{\psi}_N \|.
\]
By the injectivity of $C_X$, there exists $c_1$ such that
\[
c_1 \| \hat{\psi} - \hat{\psi}_N  \| \leq \| C_X \hat{\psi} - C_X \hat{\psi}_N \| = \| \Tilde{\psi} - \Tilde{\psi}_N \| \leq \| C_{XX^+} - C_{XX^+}^N \| \| \psi \| \leq \varepsilon \| \psi \|,
\]
from which we obtain \( \displaystyle\| C_X^{-1} C_{XX^+} \psi - C_X^{-1} C_{XX^+}^N \psi \| \leq \frac{\varepsilon}{c_1} \| \psi \|\).
\\
\\
Now, for the term  \(\| C_X^{-1} C_{XX^+}^N \psi - \left (C_X^N \right )^{-1} C_{XX^+}^N \psi \|\) denote \(\Tilde{\psi}_N = C_{XX^+}^N \psi\) so that  
\[
\| C_X^{-1} C_{XX^+}^N \psi - \left (C_X^N \right )^{-1} C_{XX^+}^N \psi \| = \| C_X^{-1} \Tilde{\psi}_N - \left (C_X^N \right )^{-1} \Tilde{\psi}_N \|.
\]
Note that  \(
\Tilde{\psi}_N \in 
\H_N = \text{Ran}(C_{X}^N) = \text{Ran}(C_{XX^+}^N) \subseteq \text{Ran}(C_{XX^+}) \subseteq \text{Ran}(C_X),
\) so that $\Tilde{\psi}_N \in \text{Ran}(C_{X}^N)$ and $\Tilde{\psi}_N \in \text{Ran}(C_{X})$. Thus, there exist $\hat{\psi}_N$ and $\hat{\psi}$ such that \(C_X^N \hat{\psi}_N = \Tilde{\psi}_N, \, C_X \hat{\psi} = \Tilde{\psi}_N,\)
so that  
\[
\| C_X^{-1} C_{XX^+}^N \psi - \left (C_X^N \right )^{-1} C_{XX^+}^N \psi \| = \| \hat{\psi}_N - \hat{\psi}\|,
\]
then, since $C_X^N \hat{\psi}_N = C_X \hat{\psi}$, it follows that  \(C_X \hat{\psi} - C_X \hat{\psi}_N = C_X^N \hat{\psi}_N - C_X \hat{\psi}.\) By the injectivity of \( C_X \), we have  
\[
c_1 \| \hat{\psi} - \hat{\psi}_N \| \leq  \| C_X \hat{\psi} - C_X \hat{\psi}_N \| \leq \| C_X^N \hat{\psi}_N - C_X \hat{\psi}_N \| \leq \| C_X^N - C_X\| \| \hat{\psi}_N \| \leq \varepsilon \| \hat{\psi}_N \|,
\]
that is \(\displaystyle \| \hat{\psi} - \hat{\psi}_N \| \leq \frac{\varepsilon}{c_1} \| \hat{\psi}_N \|.\)
\\
\\
On the other hand, \(C_X^N \hat{\psi}_N = C_X \hat{\psi}_N + (C_X^N - C_X) \hat{\psi}_N,\) and therefore \(-C_X \hat{\psi}_N = -C_X^N \hat{\psi}_N + (C_X^N - C_X) \hat{\psi}_N,\) which implies  
\[
\| C_X \hat{\psi}_N \| \leq \| C_X^N \hat{\psi}_N \| + \| (C_X^N - C_X) \hat{\psi}_N \| \leq \| C_X^N \hat{\psi}_N \|+ \| C_X^N - C_X \| \| \hat{\psi}_N \| \leq \| C_X^N \hat{\psi}_N \| + 
\varepsilon \| \hat{\psi}_N \|,
\]
from which it follows that \(\| C_X \hat{\psi}_N \| - \varepsilon \| \hat{\psi}_N \| \leq \| C_X^N \hat{\psi}_N \|.\) By the injectivity of \( C_X \), we have \(c_1 \| \hat{\psi}_N \| - \varepsilon \| \hat{\psi}_N \|  \leq \| C_X^N \hat{\psi}_N \|\).
Recalling that \( C_X^N \hat{\psi}_N = \Tilde{\psi}_N \) and \( \Tilde{\psi}_N = C_{XX^+}^N \psi \), it follows that  
\[
(c_1 - \varepsilon) \| \hat{\psi}_N \| \leq \| C_X^N \hat{\psi}_N \| = \| \Tilde{\psi}_N \| = \| C_{XX^+}^N \psi \| \leq \| C_{XX^+}^N \| \| \psi \|.
\]
If \( \varepsilon < c_1 \), we obtain \(\displaystyle \| \hat{\psi}_N \| \leq \frac{1}{c_1 - \varepsilon}
(\| C_{XX^+} \| + \| C_{XX^+} - C_{XX^+} \| ) \| \psi \|\). With all this, we obtain  
\[
\begin{aligned}
    \| \hat{\psi} - \hat{\psi}_N \| \leq \frac{\varepsilon}{c_1(c_1 - \varepsilon)}
(\| C_{XX^+} \| + \| C_{XX^+} - C_{XX^+} \| ) \| \psi \| & \leq \frac{\varepsilon}{c_1^2}
(\| C_{XX^+} \| + c_1 ) \| \psi \|.
\end{aligned}
\]
Thus, \(\displaystyle \| C_X^{-1} C_{XX^+}^N \psi - \left (C_X^N \right )^{-1} C_{XX^+}^N \psi \| \leq \frac{\varepsilon}{c_1^2}
(\| C_{XX^+} \| + c_1 ) \| \psi \|\) and therefore,  
\[
\| \U \psi - \U_N \psi \| \leq \frac{\varepsilon}{c_1} \| \psi \| + \frac{\varepsilon}{c_1^2}
(\| C_{XX^+} \| + c_1 ) \| \psi \|,
\]
from which we conclude that  
\[
\| \U - \U_N \| \leq \left ( \frac{1}{c_1} + \frac{1}{c_1^2}
(\| C_{XX^+} \| + c_1 ) \right ) \varepsilon \leq \left ( \frac{1}{c_1} + \frac{1}{c_1^2}
(\sqrt{\| k \|_{\infty}} + c_1 ) \right ) \varepsilon.
\]
We have that 
\[
\delta > 2\text{exp} \left ( - \frac{N c_1^2}{8 \| k \|_\infty} \right ) \implies \sqrt{8 \| k \|_{\infty} \ln \left ( \frac{2}{\delta} \right )} < N^{1/2} c_1.
\]
Defining \(\varepsilon = \sqrt{8 \| k \|_{\infty} \ln \left ( \frac{2}{\delta} \right )} N^{-1/2} < c_1 \), we obtain  
\[
\| \U - \U_N \| \leq \left ( \frac{2}{c_1} + \frac{\sqrt{\| k \|_{\infty}}}{c_1^2}
  \right )\sqrt{8 \| k \|_\infty \ln \left ( \frac{2}{\delta}\right ) } N^{-1/2}.
\]
\end{proof}

\subsection{Remarks of Theorem \ref{theo:error_koop_sqrt_N_def}}

The bound established by Philipp et al.~\cite{Philipp2024ErrorOperator}, which serves as the primary point of comparison for this work, is presented in the following theorem. Although their result pertains to stochastic systems in continuous time, it concerns the approximation of the Koopman operator---not the associated evolution semigroup. Nevertheless, the underlying definitions and analytical techniques are analogous.

\begin{theo*}[Philipp et al.~\cite{Philipp2024ErrorOperator}] \label{teo:error_koop}
    Let \( N \in \mathbb{N} \) be arbitrary. Assume that the first \( N + 1 \) eigenvalues \( \lambda_j \) of \( C_X \) are simple, i.e., \( \lambda_{j+1} < \lambda_j \) for all \( j = 1, \ldots, N \). Define:
    \[
    \delta_N = \min_{j=1, \ldots, N} \frac{\lambda_j - \lambda_{j+1}}{2}, \quad \| k \|_1 = \int_{\X}  |k(x,x)| \, d\mu_\X (x), \quad c_N = \frac{1}{\sqrt{\lambda_N}} + \frac{N + 1}{\delta_N \lambda_N} (1 + \|k\|_{1}) \|k \|^{1/2}_{1}.
    \]
    Let \( \varepsilon \in (0, \delta_N) \) and \( \delta \in (0, 1) \) be arbitrary, and assume that
    \[
    N \geq  \frac{8\|k\|^2_\infty \ln(2/\delta)}{\varepsilon^2}.
    \]
    If \( \U \H \subset \H \), then with probability at least \( (1 - \delta)^2 \), the following bound holds:
    \[
    \|\U - \U_N \|_{\H_N \to L^2(\X; \mu_\X)} \leq \sqrt{\lambda_{N+1}} \|\U \|_{\H_{X} \to \mathcal{H}_\X} + c_N \varepsilon.
    \]
\end{theo*}

The proof of this result employs techniques similar to those used in the present work, in the sense that Hoeffding's inequality plays a central role in deriving a probabilistic bound. However, the authors rely on the simplicity of the eigenvalues of \( C_X \), which allows for bounds involving the spectral gap. This assumption introduces a significant limitation: verifying the simplicity of the eigenvalues in practical applications is often infeasible.

Additionally, suppose we consider the case where the right-hand side of the inequality is \( O(N^{-1/2}) \). To ensure this convergence rate, the success probability must be severely restricted. Moreover, the bound does not exhibit asymptotic consistency in probability: it does not guarantee that the probability of success tends to 1 as \( N \to \infty \).

It is important to highlight some aspects regarding the bound in this work. First, it is crucial to note that not much is known about the constant \( c_1 \), except that it must be bounded by the norm of the operator \( C_X \), as shown by the inequality
\[
c_1 \| \psi \| \leq \| C_X \psi \| \leq \| C_X \| \| \psi \| \leq \sqrt{\| k \|_\infty} \| \psi \|, \quad \forall \psi \in \mathcal{H} \setminus \{ 0 \},
\]
which implies \( c_1 \leq \sqrt{\| k \|_\infty} \). While this is not particularly useful for the bounds proven here, it is helpful for determining the maximum probability of success given an \( N \in \mathbb{N} \). Indeed, we have
\[
2\exp \left( - \frac{N c_1^2}{8 \| k \|_\infty} \right) \geq 2\exp \left( - \frac{N \| k \|_\infty}{8 \| k \|_\infty} \right) = 2\exp \left( - \frac{N}{8} \right),
\]
which leads to the maximum probability of success given by
\[
\delta_{\text{max}} = \left( 1 - 2\exp \left( - \frac{N}{8} \right) \right)^2,
\]
and it converges to 1 as \( N \to \infty \).

Another relevant point is that, for some \( N > 1 \), there always exists a \( \delta \) that satisfies the condition required by the theorem. This condition is
\[
\delta_{\text{adm}, N} = 2\exp \left( - \frac{(N-1) c_1^2}{8 \| k \|_\infty} \right),
\]
which, although practically unknown due to the fact that \( c_1 \) is not explicitly known, exactly describes the exponential decay of the probability of failure as \( N \to \infty \).

Additionally, a lower bound for the constant \( C_\delta \) can be derived, and its growth can be estimated as
\[
\Tilde{C}_\delta = \frac{3}{\sqrt{\| k \|_\infty}} \sqrt{8 \| k \|_\infty \ln \left( \frac{2}{\delta} \right)} \leq \left( \frac{2}{c_1} + \frac{\sqrt{\| k \|_\infty}}{c_1^2} \right) \sqrt{8 \| k \|_\infty \ln \left( \frac{2}{\delta} \right)} \leq C_\delta.
\]
As shown in Table~\ref{tab:ctes_cota_kEDMD}, the growth of this constant is relatively slow, characteristic of logarithmic growth within a square root, while the probability of success grows exponentially, considering \( \delta_{\text{adm}, N} \).

\begin{table}[]
    \centering
    \begin{tabular}{|c|c|c|}
\hline
\( N \) & \( (1 - \delta_{\text{adm}, N})^2 \cdot 100 \% \) & \( \Tilde{C}_{\delta_{\text{adm}, N}} \) \\ \hline
10 & 25.97\% & 4.50 \\ \hline
50 & 32.20\% & 10.50 \\ \hline
100 & 82.68\% & 14.92 \\ \hline
200 & 99.20\% & 21.16 \\ \hline
300 & 99.96\% & 25.94 \\ \hline
\end{tabular}
    \caption{Values of \( (1 - \delta_{\text{adm}, N})^2 \cdot 100 \% \) and \( \Tilde{C}_{\delta_{\text{adm}, N}} \) for different values of \( N \). Assumed that \( \| k \|_\infty = 1 \) (for the Matérn kernel) and \( c_1 = 0.5 \).} 
  \label{tab:ctes_cota_kEDMD}
\end{table}

\begin{figure}
    \centering
\includegraphics[width=0.7\linewidth]{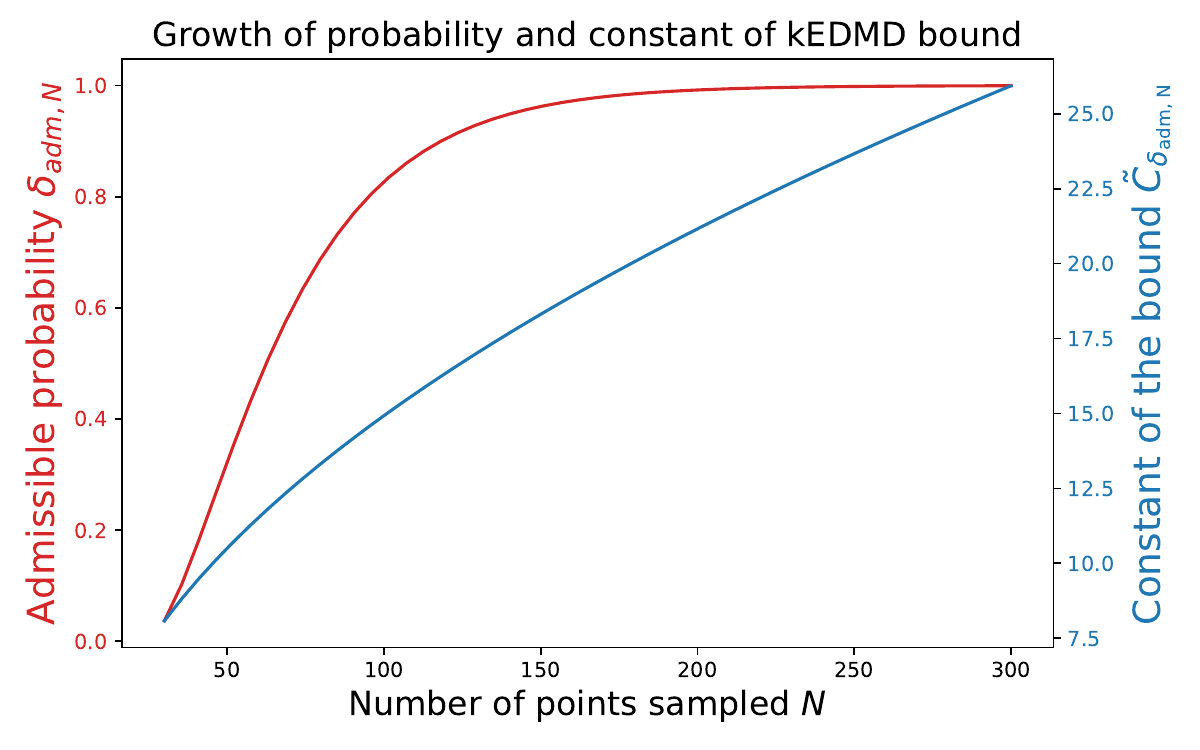} 
    \caption{Plot of the growth of the admissible probability of success \( (1 - \delta_{\text{adm}, N})^2 \) and the constant \( \Tilde{C}_{\delta_{\text{adm}, N}} \).} 
  \vspace{0.5em} 
\end{figure}

\subsection{Proof of Theorem \ref{theo:lifted_traj}}

Given a bound for the Koopman operator, we can straightforwardly generate a bound for lifted trajectories. Once the embedding of an initial point is done in \( \mathcal{H} \), the system evolution can be represented linearly, i.e.,
\begin{equation}
    \label{eq:mean_exact}
    \hat{\mu}_{k+1} = \U^* \hat{\mu}_k, \quad \hat{\mu}_0 = \Phi(\mathbf{x}_0).
\end{equation}
In the stochastic setting, this will only recover the mean trajectory, due to the properties of the conditional embedding operator. For the first evolution, we have
\[
\hat{\mu}_1 = C_{X^+|X} \Phi(\mathbf{x}_0) = \mathbb{E}[\Phi(X^+)|X = \mathbf{x}_0] = \mathbb{E}[\Phi(\mathbf{x}_1)],
\]
and for the second
\[
\hat{\mu}_2 = C_{X^+|X} \hat{\mu}_1 = C_{X^+|X} \mathbb{E}[\Phi(\mathbf{x}_1)] = \mathbb{E}[C_{X^+|X} \Phi(\mathbf{x}_1)] = \mathbb{E}[\mathbb{E}[\Phi(X^+)|X = \mathbf{x}_1]] = \mathbb{E}[\mathbb{E}[\Phi(\mathbf{x}_2)]].
\]
Iteratively, we have \( \hat{\mu}_k = \mathbb{E}[\dots \mathbb{E}[\mathbb{E}[\Phi(\mathbf{x}_k)]]] \).

However, a realization of the trajectory can also be embedded, following a linear and stochastic dynamic in \( \mathcal{H} \). This is given by \( \mu_k = \Phi(\mathbf{x}_k) \), and we can add a null term to the recurrence relation for the term \( k+1 \):
\[
\mu_{k+1} = C_{X^+|X} \Phi(\mathbf{x}_k) + \Phi(\mathbf{x}_{k+1}) - C_{X^+|X} \Phi(\mathbf{x}_k),
\]
which leads to
\begin{equation}
    \label{eq:stoch_exact}
    \mu_{k+1} = \U^* \mu_k + \zeta_k, \quad \mu_0 = \Phi(\mathbf{x}_0),
\end{equation}
where \( \zeta_k = \Phi(\mathbf{x}_{k+1}) - C_{X^+|X} \Phi(\mathbf{x}_k) \) is a zero-mean random variable because
\[
\begin{aligned}
    \mathbb{E}[\zeta_k] 
    &= \mathbb{E}[\Phi(X^+)|X = \mathbf{x}_k] - \mathbb{E}[C_{X^+|X} \Phi(\mathbf{x}_k)] \\
    &= C_{X^+|X} \Phi(\mathbf{x}_k) - C_{X^+|X} \Phi(\mathbf{x}_k) = 0,
\end{aligned}
\]
where the measurability of \( C_{X^+|X} \Phi(\mathbf{x}_k) \) with respect to the expectation of \( \zeta_k \) and Kernel Bayes' Rule are applied.

Therefore, a \textit{lifting-back} operator is required to recover the trajectories generated in the lifted space. This operator \( \B : \mathcal{H} \to \mathbb{R}^n \) must satisfy \( \B \Phi(\mathbf{x}) = \mathbf{x} \), i.e., \( \B \) acts as a left inverse of the feature map. In fact, this operator can be seen as a covariance operator. For this purpose, we define a new kernel \( \Tilde{k}: \X \times \X \to \mathbb{R} \) with RKHS \( \Tilde{\mathcal{H}} \), which acts as the dot product kernel, i.e., \( \Tilde{k}(x, y) = \langle x, y \rangle \), with the associated canonical feature map \( \Tilde{\phi}: \X \to \Tilde{\mathcal{H}} \), where \( \Tilde{\phi}(x) = \Tilde{k}(x, \cdot) = \langle x, \cdot \rangle \). This gives the RKHS \( \Tilde{\mathcal{H}} = (\mathbb{R}^n)^* \), the dual of \( \mathbb{R}^n \).

Figure~\ref{fig:KoopDiag} illustrates the entire process for obtaining trajectories in the original space. Using this, we can define a new covariance operator.
\begin{figure}[ht]
    \centering
    \includegraphics[width=0.8\linewidth]{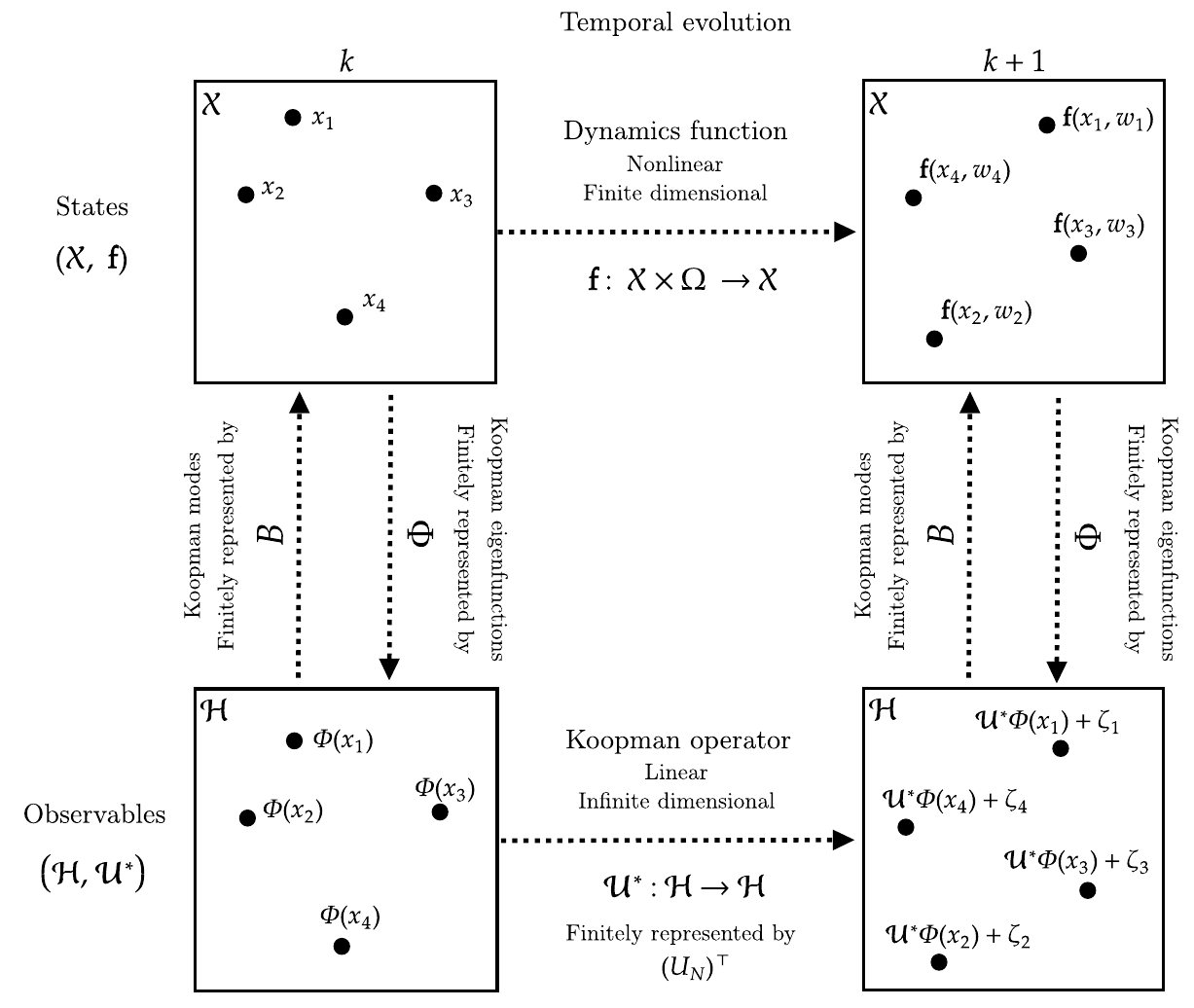}
    \caption{Diagram explaning Extended Dynamic Mode Decomposition with this notation. Adapted from Williams et al. \cite{Williams2015ADecomposition}.}
    \label{fig:KoopDiag}
\end{figure}
\begin{defn}
    We define the covariance operator \( C_{XX}: \mathcal{H} \to \Tilde{\mathcal{H}} \) as
    \[
    C_{XX} = \mathbb{E}[\Tilde{\phi}(X) \otimes \Phi(X)],
    \]
    and the conditional embedding operator \( C_{X|X} = C_{XX} C_X^{-1} \).
\end{defn}

By the Kernel Bayes' Rule, this operator acts as a left inverse of the feature map, thereby defining the desired lifting-back operator.

\begin{prop} \label{prop:left_inverse_feature}
    \( C_{X|X}\Phi(\mathbf{x}) = \mathbf{x} \), \(\forall \mathbf{x} \in \mathcal{X} \).
\end{prop}

\begin{proof}[Proof of Proposition \ref{prop:left_inverse_feature}] 
    This follows straightforwardly from the Kernel Bayes' Rule:
    \[
    \mathbf{x} = \mathbb{E}[\Tilde{\phi}(X) | X = \mathbf{x}] = C_{X|X} \Phi(\mathbf{x}).
    \]
\end{proof}

Thus, we define \( \mathcal{B} = C_{X|X} \), and by the previous section, \( \mathcal{B}^* \) can be viewed as a Koopman operator, which can be approximated by finite-range operators. Given the same set of points from the previous section \( \{x_i\}_{i=1}^N \), we define
\[
C_{XX}^N = \frac{1}{N} \sum_{j=1}^N \Phi(x_i) \otimes \Phi(x_i),
\]
and the approximation of \( \mathcal{B} \) as \( \mathcal{B}_N: \mathcal{H} \to \Tilde{\mathcal{H}} \) is given by \(\mathcal{B}_N = C_{XX}^N (C_X^N)^{-1}\).

Now, let the following trajectories correspond to those generated by kEDMD for the mean trajectory \eqref{eq:mean_exact} and the trajectory with additive noise \eqref{eq:stoch_exact}, respectively:
\begin{equation*}
    \label{eq:mean_app}
    \tag{\ref{eq:mean_exact}.1}
    \hat{\mu}_{N,k+1} = \U^*_N \hat{\mu}_{N,k}, \quad \hat{\mathbf{x}}_{N,k} = \mathcal{B}_N \hat{\mu}_{N,k}, \quad \hat{\mu}_{N,0} = \Phi(\mathbf{x}_0),
\end{equation*}
\begin{equation*}
    \label{eq:stoch_app}
    \tag{\ref{eq:stoch_exact}.1}
    \mu_{N,k+1} = \U^*_N \mu_{N,k} + \hat{\zeta}_{N,k}, \quad \mathbf{x}_{N,k} = \mathcal{B}_N \mu_{N,k}, \quad \mu_{N,0} = \Phi(\mathbf{x}_0),
\end{equation*}
with \( \hat{\zeta}_{N,k} \) as an approximation of \( \zeta_k \) by the mean, given by
\[
\hat{\zeta}_{N,k} = \frac{1}{N_\zeta} \sum_{j=1}^{N_\zeta} \left( \Phi(\mathbf{f}(\mathbf{x}_{N,k}, \mathbf{w}_k^j)) - C_{X^+|X} \Phi(\mathbf{x}_{N,k}) \right),
\]
where \( \{\mathbf{w}_k^j\}_{j=1}^N \) are realizations of the stochastic noise in the dynamics. Using these definitions and Theorem \ref{theo:error_koop_sqrt_N_def}, it can be proven that there is an error bound of order \( O(N^{-1/2}) \) for the lifted-back trajectories.

\begin{proof}[Proof of Theorem \ref{theo:lifted_traj}]
    First, as in Proposition \ref{prop:app_covariances}, for \( \varepsilon > 0 \), with probability \( 1 - \delta \), we have
    \[
    \| C_{XX} - C_{XX}^N \| \leq \varepsilon,
    \]
    with
    \[
    \delta > 2 \exp \left( - \frac{N \varepsilon^2}{8 \| k \|_\infty} \right).
    \]
    Then, similarly to Theorem \ref{theo:error_koop_sqrt_N_def}, if \( N \) and \( \delta \) satisfy the previous inequality and \( \U\H \subseteq \H \), \( \B^*(\R^n)^* \subseteq \H \), then with probability \( (1 - \delta)^3 \),
    \[
    \| \U - \U_N \|_{\mathcal{H} \to \mathcal{H}} \leq C_\delta N^{-1/2}, \quad \| \B^* - \B_N^* \|_{(\R^n)^* \to \mathcal{H}} \leq C_\delta N^{-1/2},
    \]
    where
    \[
    C_\delta = \left( \frac{2}{c_1} + \frac{\sqrt{\| k \|_\infty}}{c_1^2} \right) \sqrt{8 \| k \|_\infty \ln \left( \frac{2}{\delta} \right)}.
    \]
    Now, for every \( \mathbf{z} \in \mathcal{X} \), we have
    \[
    \| \Phi(\mathbf{z}) \|^2 \leq \max_{\mathbf{x} \in \mathcal{X}} \| \Phi(\mathbf{x}) \|^2 = \max_{\mathbf{x} \in \mathcal{X}} \langle \Phi(\mathbf{x}), \Phi(\mathbf{x}) \rangle = \max_{\mathbf{x} \in \mathcal{X}} k(\mathbf{x}, \mathbf{x}) = \| k \|_\infty.
    \]
    Therefore, in particular, \( \| \hat{\mu}_k \| \leq \sqrt{\| k \|_\infty} \) for all \( k \).

    To prove the bound for \( \| \hat{\mathbf{x}}_k - \hat{\mathbf{x}}_{N,k} \| \), we first show by induction that there exists \( \Tilde{C}_{\delta,k}^1 \) such that
    \[
    \| \hat{\mu}_k - \hat{\mu}_{N,k} \| \leq \Tilde{C}_{\delta,k}^1 N^{-1/2}.
    \]
    For \( k = 1 \), we have
    \[
    \| \hat{\mu}_1 - \hat{\mu}_{N,1} \| = \| \U^* \hat{\mu}_0 - \U^*_N \hat{\mu}_{N,0} \| \leq \| \U^* - \U^*_N \| \| \Phi(\mathbf{x}_0) \| \leq C_\delta \sqrt{\| k \|_\infty} N^{-1/2}.
    \]
    Thus, for \( k = 1 \), we have \( \Tilde{C}_{\delta,1}^1 = C_\delta \sqrt{\| k \|_\infty} \).

    Now, suppose the bound holds for some \( k \in \mathbb{N} \). We need to prove it for \( k + 1 \). First, note that
    \[
    \| \hat{\mu}_{k+1} - \hat{\mu}_{N,k+1} \|
    \leq \| \U^* \hat{\mu}_k - \U^*_N \hat{\mu}_k \| + \| \U^*_N \hat{\mu}_k - \U^*_N \hat{\mu}_{N,k} \|.
    \]
    Using the triangle inequality, we get
    \[
    \| \hat{\mu}_{k+1} - \hat{\mu}_{N,k+1} \|
    \leq \| \U - \U_N \| \| \hat{\mu}_k \| + \| \U_N \| \| \hat{\mu}_k - \hat{\mu}_{N,k} \|.
    \]
    Substituting the bounds, we obtain
    \[
    \| \hat{\mu}_{k+1} - \hat{\mu}_{N,k+1} \|
    \leq C_\delta N^{-1/2} \| \hat{\mu}_k \| + (C_\delta + \| \U \|) \Tilde{C}_{\delta,k}^1 N^{-1/2}.
    \]
    This simplifies to
    \[
    \| \hat{\mu}_{k+1} - \hat{\mu}_{N,k+1} \|
    \leq \left( C_\delta \sqrt{\| k \|_\infty} + (C_\delta + \| \U \|) \Tilde{C}_{\delta,k}^1 \right) N^{-1/2}.
    \]
    Therefore, the bound holds for \( k + 1 \) with \( \Tilde{C}_{\delta,k+1}^1 = \left( C_\delta \sqrt{\| k \|_\infty} + (C_\delta + \| \U \|) \Tilde{C}_{\delta,k}^1 \right) \).

    Now for \( \| \hat{\mathbf{x}}_k - \hat{\mathbf{x}}_{N,k} \| \), for \( k = 0 \), we have
    \[
    \begin{aligned}
        \| \hat{\mathbf{x}}_0 - \hat{\mathbf{x}}_{N,0} \| = \| \B \hat{\mu}_0 - \B_N \hat{\mu}_{N,0} \| = \| \B \Phi(\mathbf{x}_0) - \B_N \Phi(\mathbf{x}_0) \|
    & \leq \| \B - \B_N \| \| \Phi(\mathbf{x}_0) \| \\ & \leq C_\delta \sqrt{\| k \|_\infty} N^{-1/2}.
    \end{aligned}
    \]
    Thus, \( C_{\delta,0}^1 = C_\delta \sqrt{\| k \|_\infty} \).

    For \( k \geq 1 \), we have
    \[
    \| \hat{\mathbf{x}}_k - \hat{\mathbf{x}}_{N,k} \|
    \leq \| \B \hat{\mu}_k - \B_N \hat{\mu}_k \| + \| \B_N \hat{\mu}_k - \B_N \hat{\mu}_{N,k} \|.
    \]
    This gives
    \[
    \| \hat{\mathbf{x}}_k - \hat{\mathbf{x}}_{N,k} \|
    \leq \| \B - \B_N \| \| \hat{\mu}_k \| + \| \B_N \| \| \hat{\mu}_k - \hat{\mu}_{N,k} \|.
    \]
    Substituting the bounds, we get
    \[
    \| \hat{\mathbf{x}}_k - \hat{\mathbf{x}}_{N,k} \|
    \leq \left( C_\delta \sqrt{\| k \|_\infty} + (\| \B \| + C_\delta) \Tilde{C}_{\delta,k}^1 \right) N^{-1/2}.
    \]
    Therefore, the constants for \( k \geq 1 \) are given by
    \[
    C_{\delta,k}^1 = \left( C_\delta \sqrt{\| k \|_\infty} + (\| \B \| + C_\delta) \Tilde{C}_{\delta,k}^1 \right).
    \]
    Finally, for \( \| \mathbb{E}[ \mathbf{x}_k - \mathbf{x}_{N,k} ] \| \), note that \( \hat{\zeta}_k \) and \( \hat{\zeta}_{N,k} \) are centered random variables. Thus,
    \[
    \| \mathbb{E}[ \mu_{k+1} - \mu_{N,k+1} ] \| = \| \mathbb{E}[ \U^* \mu_k + \zeta_k - \U_N^* \mu_{N,k} - \hat{\zeta}_{N,k} ] \| = \| \mathbb{E}[ \U^* \mu_k - \U_N^* \mu_{N,k}] \|.
    \]
    The calculations proceed as before, concluding with the same bound.
\end{proof}

\subsection{Remarks of Theorem \ref{theo:lifted_traj}}

In the context of EDMD, other methods for recovering trajectories have been proposed, supported by empirical justification. However, to the best of our knowledge, this work represents the first attempt to generalize this idea to a broader concept of the left inverse of the feature map. In this sense, the introduction of the lifting-back operator constitutes an original contribution of this research.

An important feature of the lifting-back operator is that it can also be interpreted as a Koopman operator, but with a different perspective, as it represents a transition between two spaces: the RKHS and the original space.

It is crucial to note that all the results presented in the theorems are applicable to mean trajectories. A straightforward relationship for stochastic trajectories does not exist due to the approximation of the additive noise. In this case, Hoeffding's inequality could again be employed, but it introduces a factor of $N^{-1/4}$, which reduces the probability of successful approximation for the entire trajectory.

%% file: content/4_numericalresults.tex
We evaluate the approximation of the Koopman operator on both linear and nonlinear systems by comparing system trajectories and analyzing the associated errors, in order to assess the bound established in Theorem~\ref{theo:lifted_traj}. In all experiments, the random seed was fixed at 42 to ensure reproducibility. The code used to generate the results is available at the following \href{https://github.com/diegoolguinw/kedmd_bound.git}{GitHub repository}\footnote{\url{https://github.com/diegoolguinw/kedmd\_bound.git}}. 

\subsection*{Linear systems}

As an example of a linear system, we consider the following discrete-time dynamical system:
\[
\mathbf{x}_{k+1} = \mathbf{A}_\alpha\mathbf{x}_k + \mathbf{w}_k,
\]
where
\[
\mathbf{A}_\alpha = 
\mathbf{I}_{3\times 3}
+ 
\begin{pmatrix}
    0.01 & 0.04 & 0 \\
    0.01 & 0.02 & \alpha \\
    0.0 & 0.4 & 0.02
\end{pmatrix},
\]
and $\mathbf{w}_k \sim \mathcal{N}(\mathbf{0}_{3}, \sigma \mathbf{I}_{3\times 3})$, with $\mathbf{I}_{3\times 3}$ denoting the $3 \times 3$ identity matrix and $\sigma = 10^{-3}$. Initially, we consider $\alpha = -0.3$. From this perspective, the discrete-time system can be interpreted as the discretization of a continuous-time system with time step $\Delta t = 1$.

For the Koopman approximation, we sample initial conditions from a normal distribution of null mean and \(\mathbf{I}_{3 \times 3}\) as covariance matrix and use a Matérn kernel with smoothness parameter $\nu = 0.5$ and length scale $\ell = 10^3$.

Next, we compare the trajectories generated by the original system and the Koopman-based approximation. We consider the initial condition $\mathbf{x}_0 = (0.1, 0.1, 0.1)$ and a time horizon of $T = 30$. The parameter used for approximating the lifted noise dynamics, $N_\zeta$, was set to 30.

Due to the stochastic nature of the system, we generate 30 sample trajectories from the dynamics and plot them along with the empirical mean trajectory. In Figure~\ref{fig:error_kEDMD_lin}, we compare the empirical error across 10 trajectory realizations with the theoretical upper bound given in Theorem~\ref{theo:lifted_traj}. Since the constant $c_1$ is unknown, we replace $C_\delta$ with an estimated value $\Tilde{C}_\delta$. The results obtained for $\delta = 10^{-15}$ show a good match with the theoretical prediction, indicating that a high-probability guarantee is necessary for the success of the approximation.

This observation is further supported by a curve fitting of the form $A \cdot N^B$, where the fitted exponent satisfies $B > 1/2$. This implies that the convergence rate observed in practice is faster than the one suggested by the theoretical bound in Theorem~\ref{theo:lifted_traj}.

\begin{figure}[!ht]
\centering
\begin{subfigure}{.48\textwidth}
  \centering
  \includegraphics[width=\linewidth]{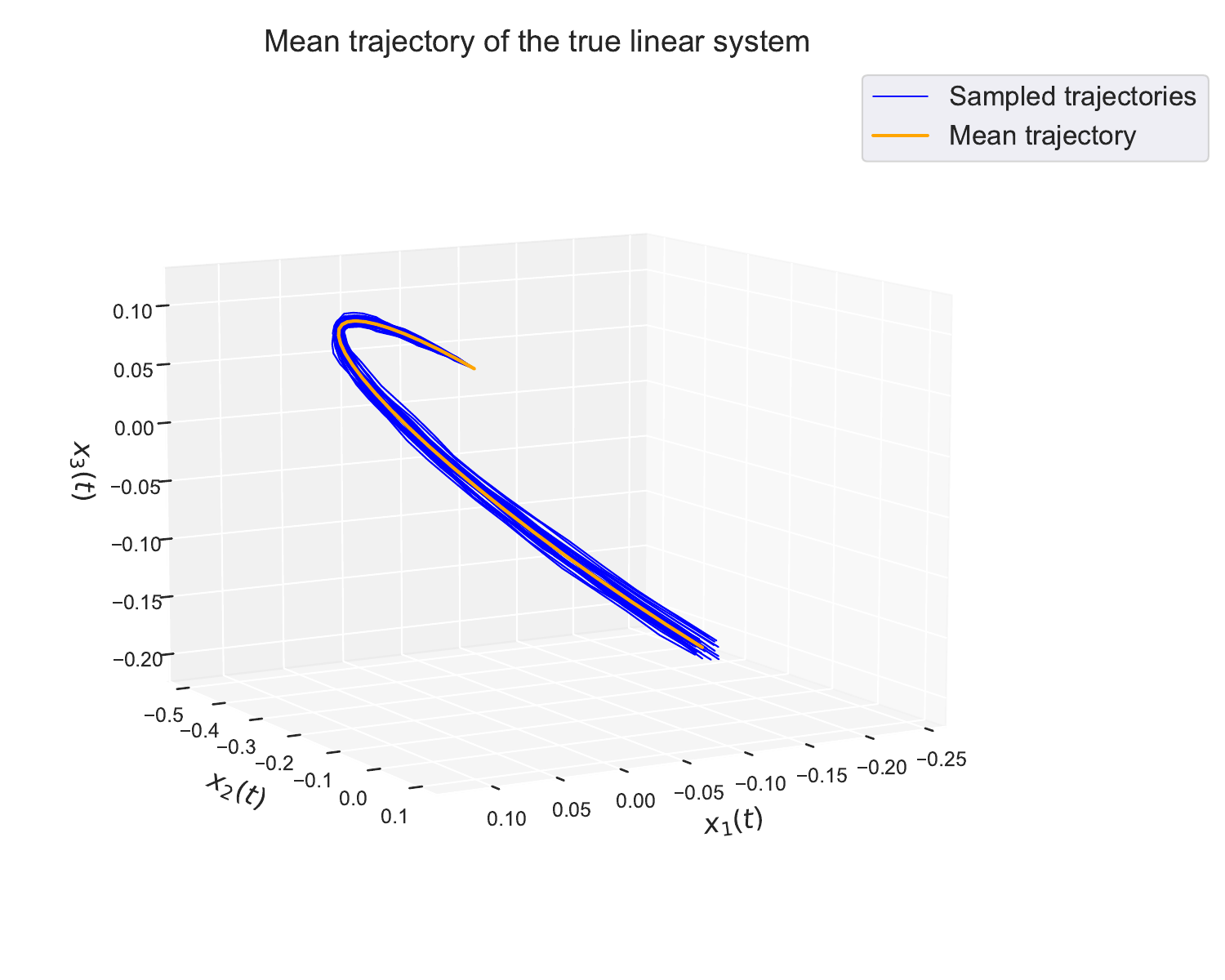}
\end{subfigure}%
\begin{subfigure}{.48\textwidth}
  \centering
  \includegraphics[width=\linewidth]{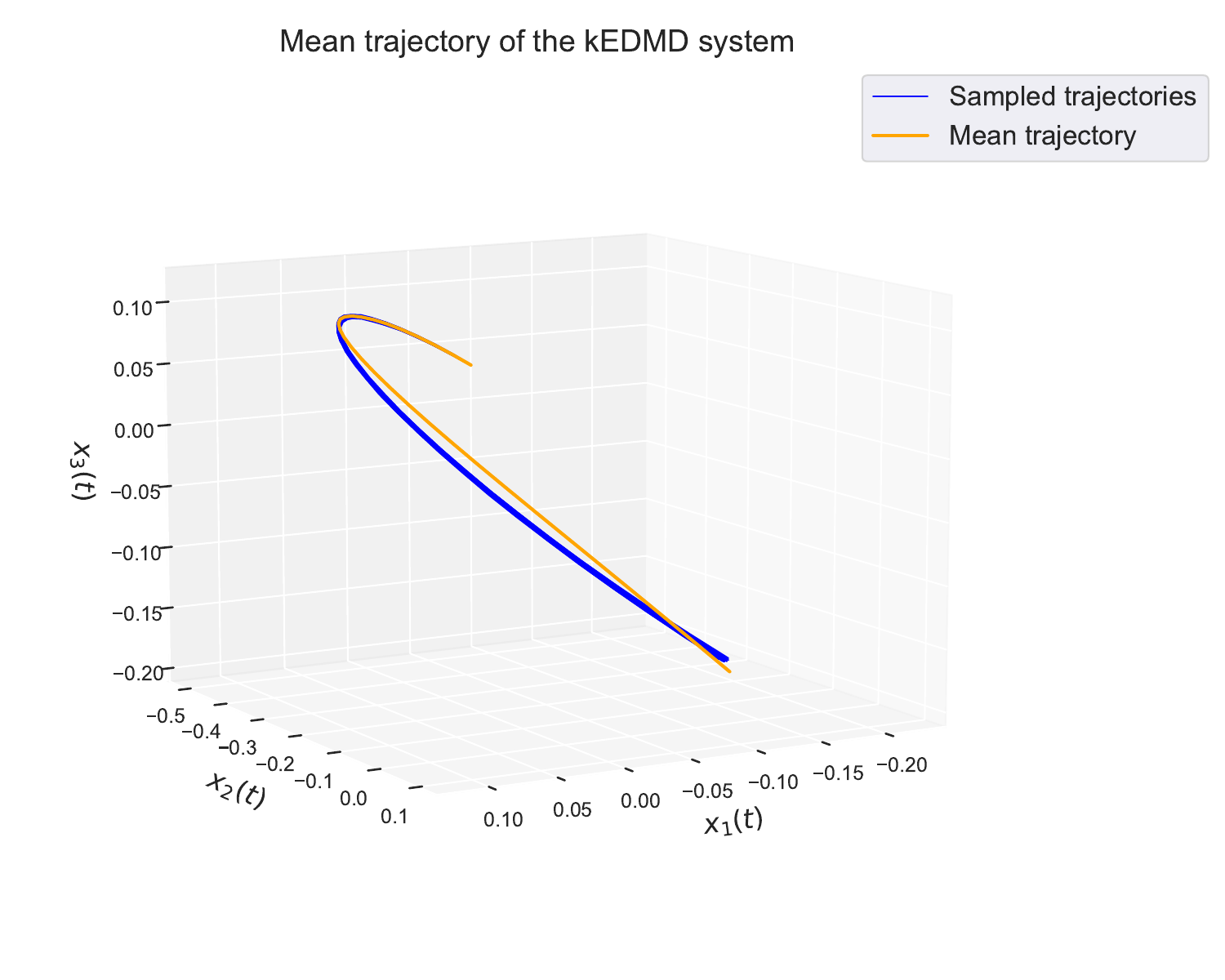}
\end{subfigure}
\caption{Phase portraits of the linear system in three-dimensional space. The orange line represents the mean trajectory, while the blue lines show 30 trajectories sampled from the system distribution.}
\label{fig:trajectories_lin_3d}
\end{figure}

\begin{figure}[!ht]
    \centering
    \includegraphics[width=\linewidth]{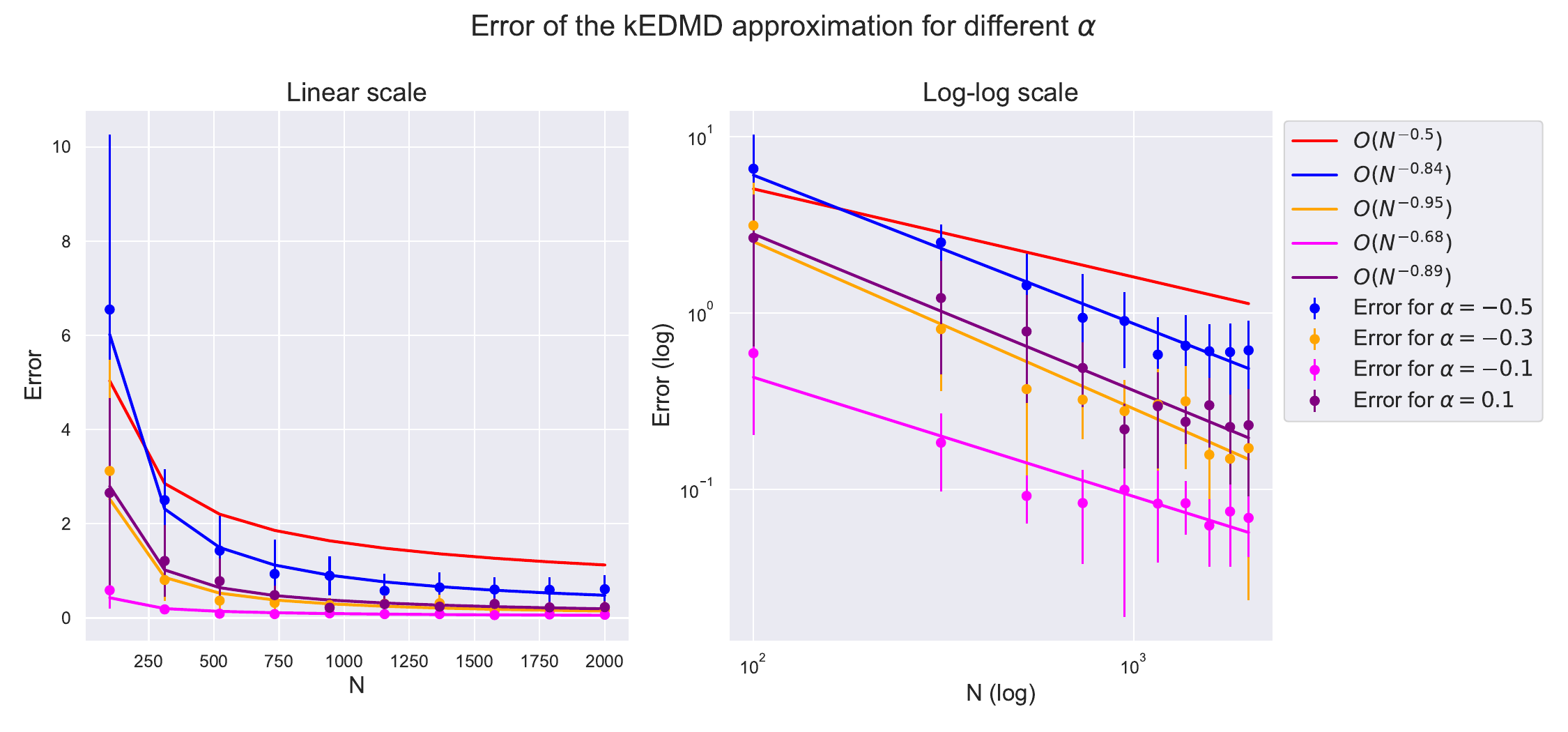}
    \caption{Errors for different values of the parameter $\alpha$ defining the system matrix $\mathbf{A}_\alpha$. The plots show the error committed by the kEDMD trajectories as a function of $N$. Dots and error bars represent 10 realizations, while the continuous lines correspond to fits of the form $A \cdot N^B$, with the fitted exponent $B$ indicated in the legend. The red line shows the theoretical upper bound $\Tilde{C}_\delta N^{-1/2}$ for $\delta = 10^{-15}$. The left panel presents the results in linear scale, and the right panel in log-log scale.}
    \label{fig:error_kEDMD_lin}
\end{figure}

\subsection*{Nonlinear Systems: SIR Model}

As an example of a nonlinear system, we test the Koopman operator approximation on the classical SIR (Susceptible-Infected-Recovered) model. This epidemiological model describes the spread of a simple infectious disease in which susceptible individuals become infected through contact with infected individuals, and subsequently recover after a certain period, without re-entering the dynamics.

The discrete-time equations of the SIR model are given by:
\[
\begin{aligned}
    S_{k+1} &= S_k - \beta S_k I_k, \\
    I_{k+1} &= I_k + \beta S_k I_k - \gamma I_k, \\
    R_{k+1} &= R_k + \gamma I_k.
\end{aligned}
\]
To account for stochasticity, we add zero-mean Gaussian noise $\mathcal{N}(\mathbf{0}_3, \sigma \mathbf{I}_{3\times 3})$ with $\sigma = 0.01$. The parameters used for the simulation are $\beta = 1.0$ and $\gamma = 0.3$. Since the state space of the SIR model lies within the 2-simplex in $\mathbb{R}^3$, we initialize the training data using samples drawn from a Dirichlet distribution with parameter vector $(1,1,1)$, which corresponds to the uniform distribution over the simplex.

\begin{figure}[!ht]
\centering
\begin{subfigure}{.48\textwidth}
  \centering
  \includegraphics[width=0.9\linewidth]{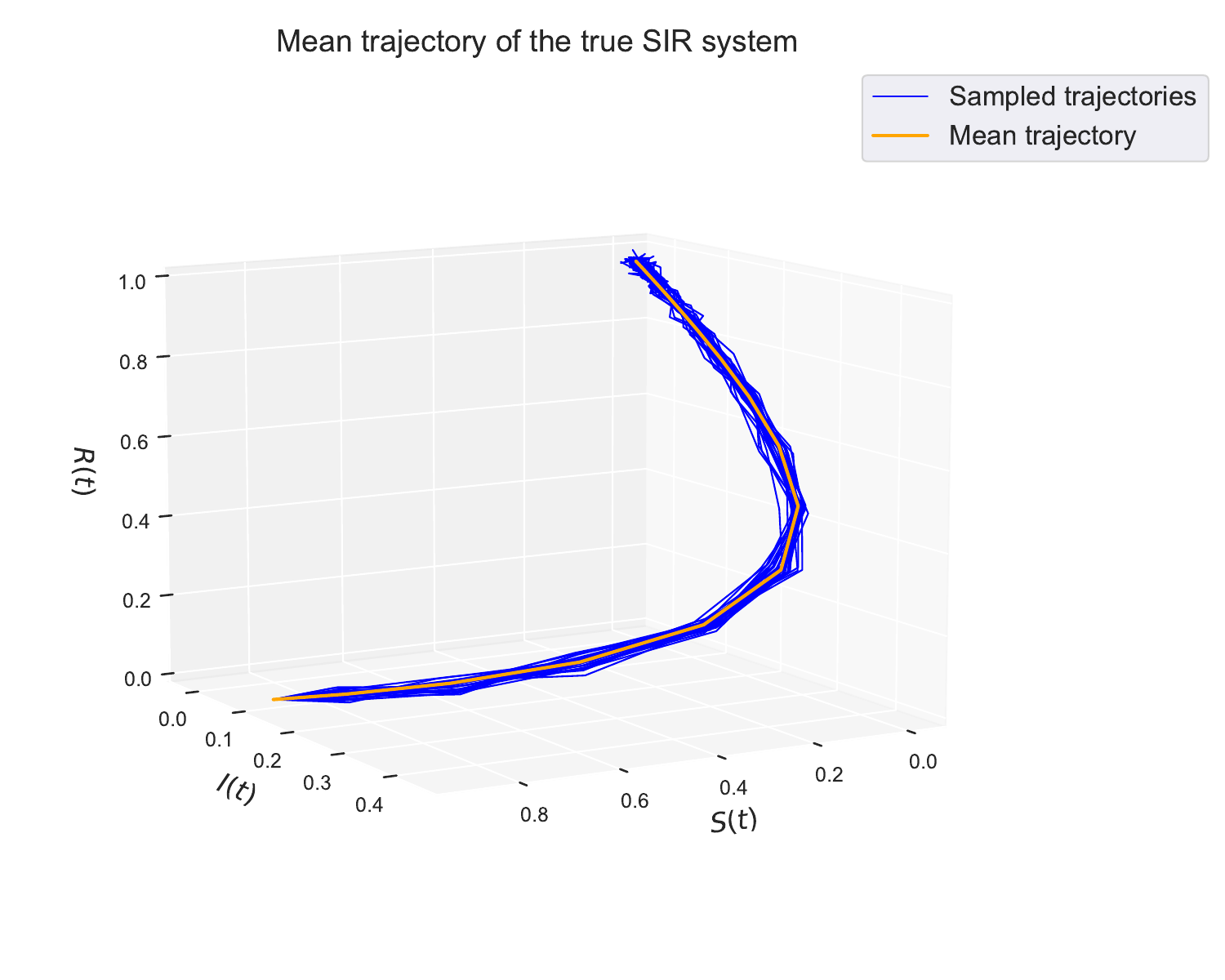}
\end{subfigure}%
\begin{subfigure}{.48\textwidth}
  \centering
  \includegraphics[width=0.9\linewidth]{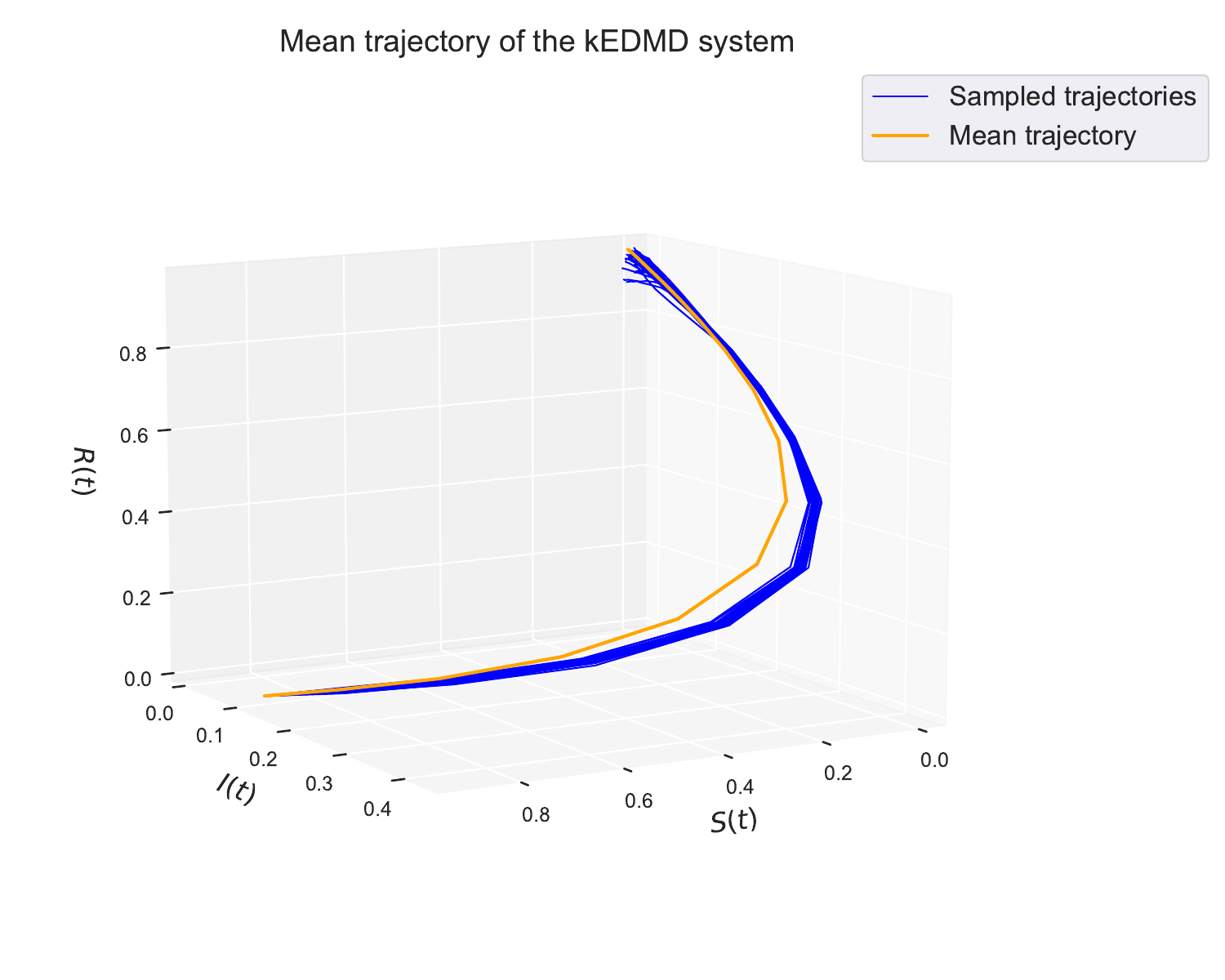}
\end{subfigure}
\caption{Phase portraits of the SIR system in three-dimensional space. The orange line represents the mean trajectory, while the blue lines correspond to 30 trajectories sampled from the system dynamics. The left panel shows the trajectories generated by the true system, and the right panel displays those generated using the kEDMD approximation.}
\label{fig:trajectories_sir_3d}
\end{figure}
\noindent The Koopman operator is approximated using a Matérn kernel with smoothness parameter $\nu = 0.5$ and length scale $\ell = 1.0$. Trajectories are generated from the initial condition $\mathbf{x}_0 = (S_0, I_0, R_0) = (0.9, 0.1, 0.0)$ over a time horizon of $T = 20$, with 30 realizations to account for the stochastic behavior. The lifted noise dynamics approximation parameter was set to $N_\zeta = 30$.
\begin{figure}[!ht]
    \centering
    \includegraphics[width=1.\linewidth]{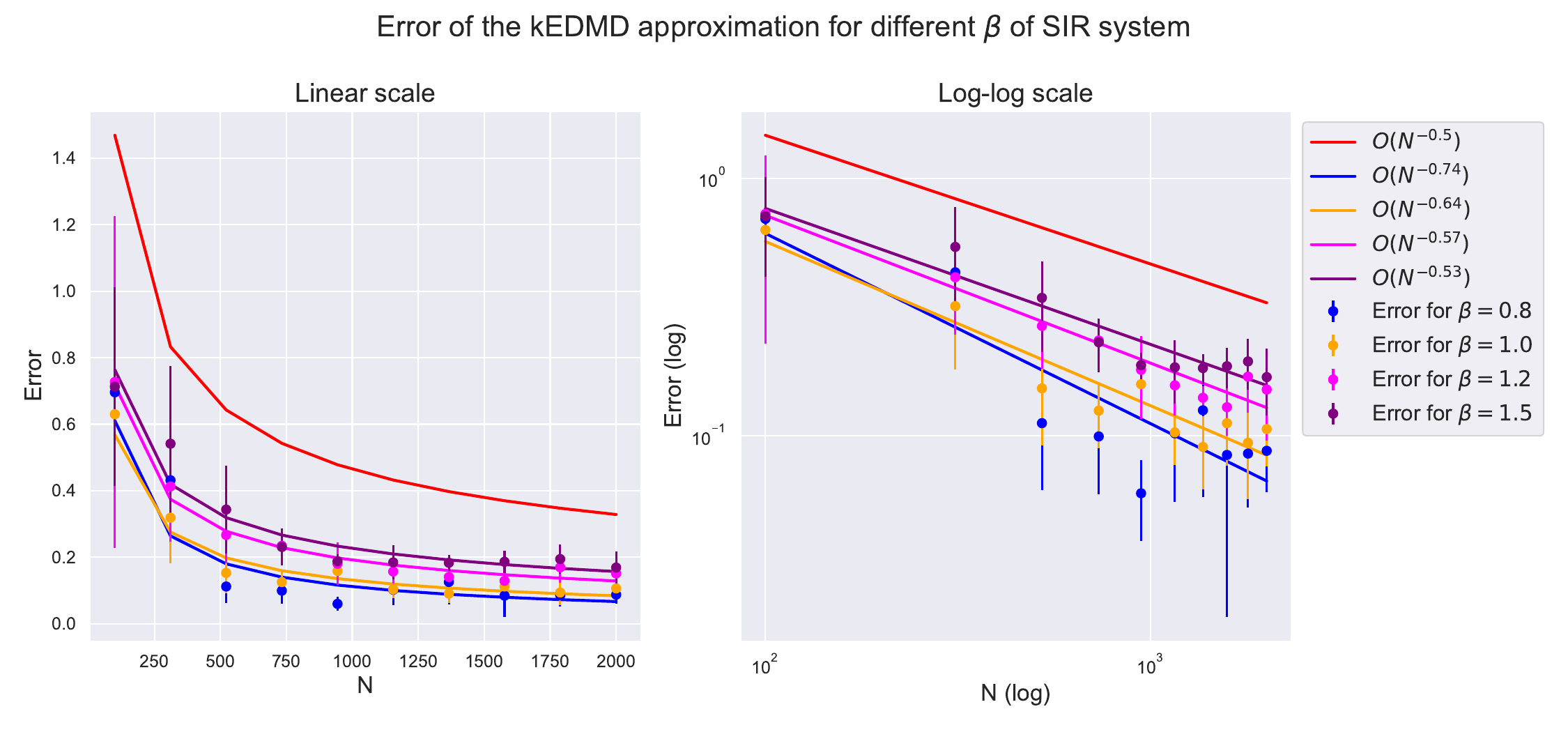}
    \caption{Errors for different values of the parameter $\beta$ defining the trasmision rate in SIR model. The plots show the error committed by the kEDMD trajectories as a function of $N$. Dots and error bars represent 10 realizations, while the continuous lines correspond to fits of the form $A \cdot N^B$, with the fitted exponent $B$ indicated in the legend. The red line shows the theoretical upper bound $\Tilde{C}_\delta N^{-1/2}$ for $\delta = 0.1$. The left panel presents the results in linear scale, and the right panel in log-log scale.}
    \label{fig:error_kEDMD_SIR}
\end{figure}
\noindent The results for trajectory comparisons between the true system and its Koopman-based approximation (kEDMD) are shown in Figure \ref{fig:trajectories_sir_3d}. These plots illustrate both individual realizations and the average system behavior across time.

In addition, Figure~\ref{fig:error_kEDMD_SIR} presents the error analysis of the kEDMD trajectories for different values of $\beta$, the parameter controlling the nonlinearity of the system. As in the linear case, we compare the empirical errors to the theoretical bound. However, for the SIR model, the constant $\Tilde{C}_\delta$ is computed with $\delta = 0.1$, indicating that this system requires a lower probability threshold for the bound to hold effectively.

%% file: content/5_discussion.tex
The bound obtained for the Koopman operator in this work is consistent with results previously reported in the literature \cite{Zhang2022APredictions, Bevanda2023KoopmanRegression, Kohne2025boldsymbolLboldsymbolinftyDecomposition}, and particularly aligned with the recent analysis in \cite{Philipp2024ErrorOperator}, which employs methods similar to those used here. However, a key distinction of the present work is the comparison of system trajectories using the lifting-back operator. This operator, which can also be interpreted as a Koopman operator, is shown to admit an approximation with an error bound of order $O(N^{-1/2})$. This result helps to unify the concept of transfer operators—both for the forward dynamics and the backward lifting—in the framework of the Koopman operator. This unification is closely related to the Perron–Frobenius operator, which governs the semigroup of the system's evolution, as previously discussed by Gerlach et al. \cite{Gerlach2020TheSystems}.

Assumptions \ref{assu:assum1}, \ref{assu:assum2} and \ref{assu:assum3} provides a sufficient condition for ensuring the inclusion $\mathcal{U}\mathcal{H} \subset \mathcal{H}$ and $(\mathcal{B})^*(\mathbb{R}^p)^* \subseteq \mathcal{H}$. These conditions are satisfied by a wide range of systems commonly encountered in engineering and scientific applications. Proposition~\ref{prop:koopman_preserve} offers a sufficient condition for the stochastic setting, extending a result previously established by Köhne et al. \cite{Kohne2025boldsymbolLboldsymbolinftyDecomposition}.

It is important to note that, for linear systems and systems driven by Gaussian noise, the assumption of compactness is generally not satisfied. However, this issue can be addressed by constraining the distribution to a compact subset with high probability. In the case of linear systems over a finite time horizon, one can also restrict attention to a compact region that contains all trajectories starting from a fixed initial condition. This strategy is applicable to a broader class of systems that do not exhibit finite-time blow-up.

Theorem~\ref{theo:lifted_traj} addresses only the expected behavior of the kEDMD estimator and does not provide any probabilistic bounds for the full stochastic system. Establishing such bounds would require additional assumptions on the dynamics function and the feature map $\Phi$, particularly the imposition of Lipschitz continuity. This would allow the derivation of bounds on the behavior of the approximated noise term $\hat{\zeta}$. Furthermore, the current bound does not fully capture the effects of the curse of dimensionality, which becomes evident when sampling points from high-dimensional spaces. Although Monte Carlo methods are typically robust to this issue, practical implementations may still encounter difficulties. One potential remedy involves incorporating bounds based on the fill distance of the sampled points, as proposed in the deterministic setting by Köhne et al. \cite{Kohne2025boldsymbolLboldsymbolinftyDecomposition}. 

Regarding the numerical results, it is observed that for both types of systems considered, the convergence rate of order $O(N^{-1/2})$ is achieved across all tested parameter values. This is particularly significant in the case of the SIR model, where the nonlinearity of the system is varied through changes in its parameters. However, an important distinction between the two systems must be emphasized. In the linear system, the fitted exponents are often smaller than $-1/2$, indicating faster convergence. This improvement in convergence comes at the cost of choosing a very small $\delta = 10^{-15}$, which leads to a success probability of $(1 - \delta)^3 \approx 1$ a highly optimistic scenario that holds only with extremely high probability.

On the other hand, for the SIR system, the fitted exponents remain close to $-1/2$, yet the bound is satisfied with a significantly lower success probability requirement. For instance, setting $\delta = 0.1$ yields a success probability of $(1 - \delta)^3 = 0.9^3 \approx 0.73$, which indicates that the theoretical bound remains valid even in experiments with moderate probability. This highlights a practical advantage of the bound in more complex, nonlinear systems such as SIR.

This can be interpreted in terms of the ability to densely fill the state space $\mathcal{X}$ with points sampled from the distribution of the state $X$. In linear systems, the state space is often large, requiring a higher number of sampled points to adequately approximate the Koopman operator. In contrast, the SIR system has a state space constrained to the 2-simplex, which is easier to fill with points drawn from a Dirichlet distribution. This difference in the structure of the state space is closely related to the probability of success: while linear systems generally require a larger sample size, nonlinear systems, such as the SIR model, need fewer points to achieve the same approximation accuracy. The exponent of the approximation error is thus tied to the inherent difficulty in approximating the system, with linear systems typically being easier to approximate than nonlinear systems like the SIR model.

In cases where some knowledge of the system's behavior near the initial condition is available, it is possible to tailor the state space distribution by sampling more points near strategic regions. Moreover, the Koopman operator can be approximated at each time step by sampling points near the past state, thereby prioritizing regions of higher likelihood. However, this strategy does not offer the same guarantees as the one-time approximation, since the bound provided is based on the assumption that Hoeffding’s inequality applies to a single sampling. Additionally, this approach increases the computational cost, as it requires multiple samplings at each time step.

Also, the hiperparameters of the kernel election could be critical in the approximation of the operator. In the case of Matérn kernel, two hyperparameters have to be choose, $\nu$, the smoothness parameter, and $\ell$ the bandwidth or lenght scale. The election of $\nu$ is driven by the regularity of the dynamical system function, in virtue of the Assumptions 2 and 3. But, the bandwidth is a mistery and, a priori, it must be decided in function of the sample structure o the number of points taken. Other kernels could be considered, for example compact support kernels \cite{Wendland1995PiecewiseDegree, Kohne2025boldsymbolLboldsymbolinftyDecomposition}, and also others possible learnable or related with neural networks as Neural Tangent Kernel \cite{Jacot2018NeuralNetworks}.

%% file: content/6_conclusion.tex
In this paper, we propose a framework for the definition and approximation of Koopman operators, which requires certain regularity conditions on the dynamics function, the distribution of the dynamics, and the domain. This framework leads to Theorem~\ref{theo:error_koop_sqrt_N_def}, which introduces a novel result for operator approximation, establishing an error bound of $O(N^{-1/2})$ with an associated probability of success. The theorem also provides an explicit constant for the error decay. However, the current error bound does not account for the curse of dimensionality, an effect that is not captured by Hoeffding’s inequality. This issue can be addressed in future work by incorporating additional terms into the bounds to account for the challenges posed by high-dimensional spaces.

This bound represents a significant contribution in itself, as it opens the door to applications involving operators acting in Hilbert spaces, particularly in Reproducing Kernel Hilbert Spaces (RKHS). This could serve as a foundational step towards the development of alternative forms of Extended Dynamic Mode Decomposition (EDMD), potentially incorporating kernel approximations via custom dictionaries, and even learnable dictionaries. For instance, this approach could be extended to neural networks, leveraging the Neural Tangent Kernel theory \cite{Jacot2018NeuralNetworks} to enhance the flexibility and adaptability of the method.

We then introduced a lifting-back operator as a method to recover trajectories in the original state space. This operator can also be viewed as a Koopman operator, yielding an error bound of $O(N^{-1/2})$ for the mean of the lifted-back trajectories. Finally, we tested the bound on several systems, observing that the errors decrease slightly faster than $O(N^{-1/2})$ in some cases, which implies that the probability of success $\delta$ varies across systems.